\documentclass[11pt]{amsart}
\usepackage{amsfonts,amssymb,amsthm}
\usepackage[numbers,sort&compress]{natbib}
\usepackage[leqno]{amsmath}
\usepackage[mathscr]{euscript}
\usepackage{graphicx} 
\usepackage{booktabs}
\usepackage{color}
\usepackage{mathrsfs}
\usepackage{paralist}
\usepackage{hyperref}
\usepackage{framed}
\usepackage{epstopdf}
\usepackage{algorithm}
\usepackage{algorithmic}
\usepackage{caption}
\newtheorem {theorem}{Theorem}[section]
\newtheorem {proposition}{Proposition}[section]

\newtheorem {lemma}{Lemma}[section]
\newtheorem {example}{Example}[section]
\newtheorem {definition}{Definition}[section]
\newtheorem {remark}{Remark}[section]
\newtheorem {assumption}{Assumption}[section]

\newcommand{\R}{{\mathbb R}}
\newcommand{\N}{\mathbb{N}}

\newcommand{\argmin}{\operatornamewithlimits{argmin}}

\makeatletter
\newcommand{\leqnomode}{\tagsleft@true}
\newcommand{\reqnomode}{\tagsleft@false}
\makeatother

\def\ees{{\accent"5E e}\kern-.385em\raise.2ex\hbox{\char'23}\kern-.08em}
\def\EES{{\accent"5E E}\kern-.5em\raise.8ex\hbox{\char'23 }}
\def\ow{o\kern-.42em\raise.82ex\hbox{
   \vrule width .12em height .0ex depth .075ex \kern-0.16em \char'56}\kern-.07em}
\def\OW{O\kern-.460em\raise1.36ex\hbox{
\vrule width .13em height .0ex depth .075ex \kern-0.16em \char'56}\kern-.07em}

\title{Multi-objective convex polynomial optimization and semidefinite programming relaxations$^*$}
\thanks{$^*$Jae Hyoung Lee was supported by the National Research Foundation of Korea (NRF) Grant funded by the Korea government (MSIP) (NRF-2018R1C1B6001842).
Nithirat Sisarat was partially supported by a grant from the Thailand Research Fund through the Royal Golden Jubilee Ph.D. Program grant number PHD/0026/2555.
Liguo Jiao was supported by Jiangsu Planned Projects for Postdoctoral Research Funds 2019 (no. 2019K151).}

\author{Jae Hyoung Lee}
\address[Jae Hyoung Lee]{Department of Applied Mathematics, Pukyong National University, Busan, 48513, Korea}
\email{mc7558@naver.com}

\author{Nithirat Sisarat}
\address[Nithirat Sisarat]{Department of Mathematics, Faculty of Science, Naresuan University, Phitsanulok 65000, Thailand}
\email{nithirats@hotmail.com}

\author{Liguo Jiao$^{\dag}$}
\address[Liguo Jiao]{School of Mathematical Sciences, Soochow University, Suzhou 215006, Jiangsu Province, China}
\email{hanchezi@163.com}
\thanks{$^{\dag}$Corresponding Author}

\date{\today}

\begin{document}

\begin{abstract}
This paper aims to find efficient solutions to a multi-objective optimization problem \eqref{modelsetting} with convex polynomial data.
To this end, a hybrid method, which allows us to transform problem~\eqref{modelsetting} into a scalar convex polynomial optimization problem~\eqref{hybridmodel} and does not destroy the properties of convexity, is considered.
First, we show an existence result for efficient solutions to problem~\eqref{modelsetting} under some mild assumption.
Then, for problem~\eqref{hybridmodel}, we establish two kinds of representations of non-negativity of convex polynomials over convex semi-algebraic sets, and propose two kinds of finite convergence results of the Lasserre-type hierarchy of semidefinite programming relaxations for problem~\eqref{hybridmodel} under suitable assumptions.
Finally, we show that finding efficient solutions to problem~\eqref{modelsetting} can be achieved successfully by solving hierarchies of semidefinite programming relaxations and checking a flat truncation condition.
\end{abstract}
\maketitle

\section{Introduction}
\label{intr}
{\leqnomode{
The multi-objective optimization is used to denote a type of optimization problems, where two or more objectives are to be minimized over certain constraints.
In this paper, we are interested in the study of a multi-objective programming problem with {\it convex} polynomial data reading as follows:
\begin{align}\label{modelsetting}
\mbox{Min$_{\R^p_+}$} f(x) \quad \mbox{ subject to } \ g_i(x) \leq0, \ i=1,\ldots,m, \tag{MP}
\end{align}
where $f(x):= \big(f_1(x),\ldots,f_p(x)\big),$ $f_j\colon \R^n \to \R,$ $j=1,\ldots,p,$ and $g_i\colon \R^n \to \R,$ $i=1,\ldots,m,$ are convex polynomials, and $\R^p_+$ stands for the non-negative orthant of $\R^p;$
moreover, ``${\rm Min}_{\mathbb{R}^p_+}$" in problem~\eqref{modelsetting} is understood with respect to the ordering cone $\mathbb{R}^p_+.$
We denote the feasible set of problem~\eqref{modelsetting} by
}}
\begin{equation}\label{feasibleset}
K:=\{x\in\R^n: g_i(x) \leq0, \ i=1,\ldots,m\},
\end{equation}
which is assumed to be nonempty (not necessarily compact) throughout this paper.

It is worth noting that, in general, there is no single optimal solution that simultaneously minimizes all the objective functions.
In these cases, one aims to look for the ``best preferred" solution, in contrast to an optimal solution.
This leads to the concept of the so-called {\it efficient} (or {\it Pareto-optimal}) solutions in multi-objective programming.
In fact, the efficient solutions are the ones that cannot be improved in one objective function without deteriorating their performance in at least one of the rest.
Now, we state the concept of (strictly) efficient solutions to problem~\eqref{modelsetting}.

\begin{definition}{\rm
A point $\bar x\in K$ is said to be an \emph{efficient solution} (known also as Pareto-optimal solution) to problem~\eqref{modelsetting} if
\begin{equation*}
f(x)-f(\bar x)\not\in-\R^p_+\backslash\{0\}, \ \ \forall x\in K.
\end{equation*}
In addition, if $\bar x$ is an efficient solution to problem~\eqref{modelsetting}, then $f(\bar x) \in \R^p$ is called a Pareto value to problem~\eqref{modelsetting},
and all Pareto values form the {\it Pareto curve}.
}\end{definition}

\subsection{Scalarization methods}
Usually, the problem~\eqref{modelsetting} can be solved (i.e., its efficient solutions could be found) by solving related single objective programming problems.
We call such a method by {\it scalarization approach}, and there are many scalarization approaches, for example, the weighted sum method, the $\epsilon$-constraint method, and the hybrid method (see~\cite{Chankong1983,Ehrgott2005}).
Below, we first describe the weighted sum method and the $\epsilon$-constraint method, and show their powerful aspects for solving some special cases of problem~\eqref{modelsetting}. Then, we minutely discuss the hybrid method, which will be used in the paper to solve problem~\eqref{modelsetting}.

[Weighted sum method].
The idea of this method is to associate each objective function with a weighting non-negative coefficient and then minimize the weighted sum of the objectives.
This method is a simple way to generate different efficient solutions, however, there are parameters as many as the number of objective functions, and this may be not easy to control with the parameters if the considered multi-objective optimization problem has a great number of objective functions.
Further, this method is effective to deal with convex cases, but for the nonconvex cases, it may go awry.
Besides, for a given non-dominated point, it is usually not easy to find a corresponding desired parameter.
In other words, it may be not easy to set parameters to obtain a non-dominated point from a desired region of the image space (i.e., the corresponding efficient solutions in the feasible set).

Nevertheless, the weighted sum method has been shown to be effective to solve some special case of problem~\eqref{modelsetting}.
For example, if the involved functions are linear (see Blanco et al.~\cite{Blanco2014}), then we call problem~\eqref{modelsetting} a {\it multi-objective linear programming} (MLP, for short).
An MLP problem has become a very important area of activity within the optimization field since the 1960s for its relevance in practice and as a mathematical topic in its own right; see \cite{Blanco2014,Luc2016}.
The development of MLPs has come in parallel to the scalar counterpart~\cite{Ploskas2017}, and its theory and algorithms have been developed continuously over the years.
It is worth noting that, in 2014, Blanco et al. \cite{Blanco2014} presented a new method to solve MLPs, and this new method was based on a transformation of any MLPs into a unique lifted semidefinite programming (SDP); however, we would emphasize here that before their new method works, the weighted sum method is used to transform the MLPs into a scalar linear programming problem.

[$\epsilon$-Constraint method].
This method consists of minimizing one of the objective functions and all the other objective functions are transformed into constraints by setting an upper bound to each of them.
Notwithstanding the fact that, in order to find efficient solutions by this method, each transformed single objective optimization problem, as many as the number of the objective functions, shall be solved; or the optimal solution to at least one single objective optimization problem should be unique.
The $\epsilon$-constraint method was still proved to be useful to solve some special case of problem~\eqref{modelsetting}.
For instance, if the involved functions are SOS-convex polynomial\footnote{A polynomial $f$ is called {\it SOS-convex} if there exists a matrix polynomial $F(x)$ such that the Hessian $\nabla^2f(x)$ factors as $F(x)F(x)^T;$ see~\cite{Ahmadi2012,Ahmadi2013,Helton2010}.}, then based on its special structures owned by any SOS-convex polynomials, in \cite{Jiao2019,Jiao2019b,Lee2018,Lee2019}, the authors employed the $\epsilon$-constraint method to transform this special case of problem~\eqref{modelsetting} into a class of scalar ones;
moreover, since the $\epsilon$-constraint method did not destroy the SOS-convex properties of a polynomial, along with these facts, an {\it exact SDP relaxation approach} was used to find the optimal solutions to the corresponding scalar problems, then efficient solutions to the special case of problem~\eqref{modelsetting} can also be found.

[Hybrid method].
In this paper, we are interested in the study of finding efficient solutions to problem~\eqref{modelsetting} with {\it convex polynomial data}, and we do this by transforming problem~\eqref{modelsetting} to a scalar one based on this {\it hybrid method}.
Mathematically speaking, in connection with problem~\eqref{modelsetting}, we consider the following (scalar) convex polynomial optimization problem \cite{Ehrgott2005,Giannessi2000}:
{\leqnomode{
\begin{align}\label{hybridmodel}
&\bar f_z:= \ \hspace{-3mm}\ \min\limits_{x\in\R^n} \ \ \lambda^Tf(x):=\sum_{j=1}^p\lambda_jf_j(x)  \tag{${\rm P}_z$} \\
& \qquad\ \ {\rm s.t.\ } \ \  g_i(x) \leq0, \ i=1,\ldots,m, \nonumber \\
& \qquad \qquad \quad f_j(x)\leq f_j(z), \ j=1,\ldots,p, \nonumber
\end{align}
where $\lambda\in{\rm int\,}\R^p_+$ is fixed and the parameter $z\in\R^n.$
Note that ${\rm int\,}\R^p_+$ stands for the interior of $\R^p_+.$
Let
}}
\begin{equation}\label{hybrid-feasibleset}
K_z:=\{x\in\R^n: g_i(x) \leq0, \ i=1,\ldots,m, \ f_j(x)\leq f_j(z), \ j=1,\ldots,p\}
\end{equation}
 be the (possibly non-compact) feasible set of problem~\eqref{hybridmodel}.
It is worth noting that $\lambda$ here does not play the role of a parameter but be fixed in problem~\eqref{hybridmodel}.
It is also worth mentioning that the feasible set $K_z$ is nonempty whenever the parameter $z$ is selected in the feasible set $K$ of problem~\eqref{modelsetting}.

Actually, the hybrid method can be regarded as the one combining the weighted sum method with the $\epsilon$-constraint method; see, for example~\cite{Ehrgott2005,Giannessi2000}.
In other words, it does not require (i) solving several problems; (ii) considering about uniqueness of an optimal solution to one single objective optimization problem; or (iii) controlling the weighting non-negative parameters.
Even though in the view of computation, this method needs more process and time than the weighted method (since the number of constraints increases), one still may find non-dominated points (corresponding to efficient solutions in the feasible set) from a desired region of the image space, which can be controlled by $\epsilon$-constraint method, in contrast to the weighted sum method.

\subsection{Convex polynomial optimization}
As mentioned in the above, the hybrid method is substantially one of the scalarization approaches, thus it is essential to recall some celebrated results in {\it scalar} (as opposed to {\it multi-objective}) polynomial optimization.

Indeed, if one restricts oneself to polynomial optimization (not necessarily convex), then one may approximate the optimal value and an optimal solution to a polynomial optimization problem as closely as desired, and sometimes solve such a polynomial optimization problem exactly; see \cite{Lasserre2001,Lasserre2009,Nie2013}.
Moreover, polynomial optimization problems have attracted a lot of attention in theoretical and applied mathematics over the years; see, for example, the related monographs~\cite{Ha2017,Lasserre2010,Lasserre2015}.
Real algebraic geometry and semi-algebraic geometry are sub-fields in algebra that are strongly related to polynomial optimization problems; see~\cite{Lasserre2015}.
Since these problems are, in general, very difficult, it is a natural choice to look for tractable relaxations.
These relaxations are often based on some variant of a ``Positivstellensatz" for given semi-algebraic sets \cite{Putinar1993,Putinar1999,Schmudgen1991}.
Many researchers have proposed hierarchies of such relaxations that are based on moment and sum-of-squares approximations of the original problem, and give semidefinite programming (SDP) problems.
For instance, under certain conditions, Lasserre \cite{Lasserre2001} proposed a hierarchy of SDP problems that converge with their optimal values to the optimal value of the original polynomial optimization problem, see also \cite{Schweighofer2005}.

The reasons why we restrict us to {\it convex} polynomial data are
(i) under convexity, the hierarchy of semidefinite programming relaxations for polynomial optimization simplifies and has finite convergence, a highly desirable feature as convex problems are in principle easier to solve; see Lasserre~\cite{Lasserre2009} for more details;
(ii) the Lasserre hierarchy of SDP relaxations with a slightly extended quadratic module for {\it convex} polynomial optimization problems always converges asymptotically even in the case of non-compact semi-algebraic feasible sets; see, \cite{Jeyakumar2014}. Moreover, as Jeyakumar {\it et al.}~\cite{Jeyakumar2014} pointed out, the positive definiteness of the Hessian of the associated Lagrangian at a saddle-point guarantees the finite convergence of the hierarchy.

\subsection{Contributions}
In 2012, Gorissen and Hertog \cite{Gorissen2012} considered a class of optimization problems with multiple {\it linear} objectives and {\it linear} constraints and used adjustable robust optimization and polynomial optimization as tools to approximate the {\em Pareto set} with polynomials of arbitrarily large degree. Later, some further results on approximating {\em Pareto curves} using semidefinite relaxations were studied by Magron et al. \cite{Magron2014}. 
In our work, we will focus on the issue on finding/approximating efficient/Pareto-optimal {\it solutions} to problem \eqref{modelsetting}.
Besides, comparing with our previous works (e.g., \cite{Jiao2019,Jiao2019b,Jiao2020,Lee2018,Lee2019}), we further  drop the assumption on the SOS-convexity of the involved polynomials, but instead by convex polynomials. 
All in all, we will make the following contributions to the multi-objective programming problem with {\it convex} polynomial data in this paper.
\begin{itemize}
\item First, we establish an existence result for efficient solutions to problem~\eqref{modelsetting} under some mild assumption.
\item Second, for problem~\eqref{hybridmodel}, we give two kinds of representations of non-negativity of convex polynomials over convex semi-algebraic sets.
\item Third, we formulate two kinds of Lasserre-type hierarchies of SDP relaxations for problem~\eqref{hybridmodel} and give finite convergence results for the hierarchy of SDP relaxations.
\item Last, under the flat truncation condition, we show a scheme on finding efficient solutions to the problem \eqref{modelsetting}.
\end{itemize}

The rest of this paper is arranged as follows.
In Sect.~\ref{sect:2}, we recall some notations and celebrated results that will be widely used throughout the paper.
We establish an existence result for efficient solutions to problem~\eqref{modelsetting} under some mild assumption in Sect.~\ref{sect:3}.
For problem~\eqref{hybridmodel}, we show two kinds of representations of non-negativity of convex polynomials over convex semi-algebraic sets; moreover, we formulate two kinds of Lasserre-type hierarchies of SDP relaxations for problem~\eqref{hybridmodel} and give their finite convergence results in Sect.~\ref{sect:4}.
Section~\ref{sect:5} provides a scheme on finding efficient solutions to problem~\eqref{modelsetting} under the flat truncation condition.
Finally, conclusions and further remarks are given in Sect.~\ref{sect:6}.

\section{Preliminaries}\label{sect:2}

We begin this section by fixing some notations and preliminaries.
We suppose $1 \leq n \in \N$ ($\N$ is the set of non-negative integers) and abbreviate $(x_1, x_2, \ldots, x_n)$ by $x.$
The Euclidean space $\R^n$ is equipped with the usual Euclidean norm $\| \cdot \|$ and the inner product $\langle \cdot, \cdot \rangle,$ here $\langle x, y \rangle :=\sum_{i = 1}^n x_iy_i,$ for any $x, y \in \R^n.$
The non-negative orthant of $\R^n$ is denoted by $\R_{+}^n.$

The space of all real polynomials in the variable $x$ is denoted by $\R[x].$
Moreover, the space of all real polynomials in the variable $x$ with degree at most $d$ is denoted by $\R[x]_d.$
The degree of a polynomial $f$ is denoted by $\deg f.$
We say that a real polynomial $f$ is sum-of-squares if there exist real polynomials $q_l,$ $l = 1,\ldots,r,$ such that $f =\sum_{l=1}^{r}q_{l}^2.$
The set consisting of all sum-of-squares real polynomials is denoted by $\Sigma[x].$
In addition, the set consisting of all sum-of-squares real polynomials with degree at most $d$ is denoted by $\Sigma[x]_{d}.$
Given polynomials $\{g_1,\ldots,g_m\} \subset \R[x],$ a {\it quadratic module} generated by the tuple $g:=(g_1,\ldots,g_m)$ is defined by
\begin{equation*}
\mathcal{Q}(g) := \left\{\sigma_0 + \sum_{i=1}^m\sigma_ig_i: \sigma_i \in \Sigma[x], \ i = 0, 1, \ldots, m\right\}.
\end{equation*}
The set $\mathcal{Q}(g)$ is \emph{Archimedean} if there exists $p\in \mathcal{Q}(g)$ such that the set $\{x \in\R^n: p(x) \geq 0\}$ is compact.

The following result is a celebrated and important representation of positive polynomials over a semi-algebraic set when the quadratic module $\mathcal{Q}(-g)$ is Archimedean.
\begin{lemma}[Putinar's Positivstellensatz]{\rm \cite{Putinar1993}}\label{Putinar}
Let $f$ and $g_i,$ $i=1,\ldots,m,$ be real polynomials on $\R^n.$
Suppose that $\mathcal{Q}(-g)$ is Archimedean.
If $f$ is strictly positive on $K:=\{x\in\R^n: g_i(x) \leq0, \ i=1,\ldots,m\}\neq\emptyset,$ then $f \in\mathcal{Q}(-g).$
\end{lemma}

For a multi-index $\alpha:=(\alpha_1,\ldots,\alpha_n)\in\N^n,$ let us denote $|\alpha|:=\sum_{i=1}^n\alpha_i$ and $\N^n_d:= \{\alpha \in \N^n :  |\alpha|\leq d\}.$
The notation $x^\alpha$ denotes the monomial $x_1^{\alpha_1}\cdots x_n^{\alpha_n}.$
The canonical basis of $\R[x]_d$ is denoted by
\begin{equation}\label{cano_basis}
v_d(x):=(x^\alpha)_{\alpha\in\N^n_{d}}=(1,x_1,\ldots,x_n,x_1^2,x_1x_2,\ldots,x_n^2,\ldots,x_1^d,\ldots,x_n^d)^T,
\end{equation}
which has dimension $s(d):=\left( \substack{ n+d \\ d }\right).$

Given an $s(2d)$-vector $y:= (y_\alpha)_{\alpha\in\N^n_{2d}}$ with $y_0=1,$ let $M_{d}(y)$ be the moment matrix of dimension $s(d),$ with rows and columns labeled by (\ref{cano_basis}) in the sense that
\begin{equation*}
M_d(y)(\alpha,\beta)=y_{\alpha+\beta}, \ \forall \alpha,\beta\in\N^n_d.
\end{equation*}
For a polynomial $x\mapsto p(x):=\sum_{\gamma\in\N_{w}^n}p_\gamma x^\gamma,$ $M_d(py)$ is the so-called localization matrix defined by
\begin{equation*}
M_d(py)(\alpha,\beta)=\sum_{\gamma\in\N_w^n}p_\gamma y_{\gamma+\alpha+\beta}, \ \forall \alpha,\beta\in\N^n_d.
\end{equation*}

For a symmetric matrix $X,$ $X \succeq 0$ (resp., $X\succ 0$) means that $X$ is positive semidefinite (resp., positive definite).
The gradient and the Hessian of a polynomial $f\in \R[x]$ at a point $\bar x$ are denoted by $\nabla f(\bar x)$ and $\nabla^2f(\bar x),$ respectively.

\section{Existence Results}\label{sect:3}
{\leqnomode{
In this section, we first establish necessary and sufficient conditions for the existence of an efficient solution to problem~\eqref{modelsetting}.

Recall the following (scalar) convex polynomial optimization problem \cite{Ehrgott2005,Giannessi2000} introduced in Sect.~\ref{intr}, which is transformed from problem~\eqref{modelsetting} by the hybrid method:
\begin{align*}
&\bar f_z:= \ \hspace{-3mm}\ \min\limits_{x\in\R^n} \ \ \lambda^Tf(x):=\sum_{j=1}^p\lambda_jf_j(x)  \tag{${\rm P}_z$} \\
& \qquad\ \ {\rm s.t.\ } \ \  g_i(x) \leq0, \ i=1,\ldots,m, \nonumber \\
& \qquad \qquad \quad f_j(x)\leq f_j(z), \ j=1,\ldots,p, \nonumber
\end{align*}
where $\lambda\in{\rm int\,}\R^p_+$ is fixed and the parameter $z\in\R^n,$
the feasible set $K_z$ of problem~\eqref{hybridmodel} is defined as \eqref{hybrid-feasibleset}.
}}

\medskip
The following proposition suggests a way to obtain an efficient solution to problem \eqref{modelsetting} by solving the problem \eqref{hybridmodel}.
\begin{proposition}{\rm \cite[Proposition~12]{Giannessi2000}}\label{proposition1}
Let $f_j\colon \R^n \to \R,$ $j=1,\ldots,p,$ and $g_i\colon \R^n \to \R,$ $i=1,\ldots,m,$ be convex polynomials$,$ consider problem~\eqref{modelsetting} and its associated scalar problem~\eqref{hybridmodel} by hybrid method.
Let $z_0\in K,$ where $K$ is defined as \eqref{feasibleset}.
If $\bar x$ is an optimal solution to problem~$({\rm P}_{z_0}),$ then $\bar x$ is also an optimal solution to problem~$({\rm P}_{\bar x}),$ and so is an efficient solution to problem~\eqref{modelsetting}.
\end{proposition}

Now, we recall a known lemma that shows an important existence result of solutions to (scalar) convex polynomial optimization problems.
\begin{lemma}\label{lemma1} {\rm \cite[Theorem~3]{Belousov2002}}
Let $f_0$ and $g_i,$ $i=1,\ldots,m,$ be convex polynomials on $\R^n.$
Let $K := \{x \in \R^n : g_i(x) \leq0, \ i = 1,\ldots,m\}.$
Suppose that $\inf_{x\in K}f_0(x) >-\infty.$
Then$,$ $\argmin_{x\in K}f_0(x)\neq\emptyset.$
\end{lemma}

As a consequence of Lemma~\ref{lemma1}, the next theorem provides necessary and sufficient conditions for the existence of an efficient solution to problem~\eqref{modelsetting}.
\begin{theorem}[Existence of Efficient Solutions]
Let $f_j\colon \R^n \to \R,$ $j=1,\ldots,p,$ and $g_i\colon \R^n \to \R,$ $i=1,\ldots,m,$ be convex polynomials$,$ consider problem~\eqref{modelsetting}, then the following statements are equivalent$:$
\begin{itemize}
  \item[{\rm (i)}] Problem \eqref{modelsetting} admits an efficient solution.
  \item[{\rm (ii)}] There exists $z_0\in\R^n$ such that $f(K)\cap \left(f(z_0)-\R^p_+\right)$ is a nonempty and bounded set.
\end{itemize}
\end{theorem}
\begin{proof}
We start by proving $[{\rm (i)}\Rightarrow{\rm (ii)}].$
To show this, let $\bar x\in K$ be an efficient solution to problem~\eqref{modelsetting}.
Then we have $f(K)\cap \left(f(\bar x)-\R^p_+\right)=\{f(\bar x)\},$ and so, the assertion ${\rm (ii)}$ holds.

Conversely, we first note that $f(K_{z_0})=f(K)\cap \big(f(z_0)-\R^p_+\big).$
From the assertion ${\rm (ii)},$ the image $f(K_{z_0})$ is nonempty and bounded.
So, there exists a positive real number $N$ such that $||f(x)||\leq N$ for all $x\in K_{z_0}.$
It then follows from the Cauchy--Schwarz inequality that for all $x\in K_{z_0},$
\begin{equation*}
\left|\sum_{j=1}^p\lambda_jf_j(x)\right|=\left|\langle \lambda, f(x)\rangle\right|\leq||\lambda||\cdot ||f(x)||\leq ||\lambda||N,
\end{equation*}
and so,  problem~$({\rm P}_{z_0})$ has a finite optimal value.
Hence, in view of Lemma~\ref{lemma1}, there exists at least one optimal solution to problem~$({\rm P}_{z_0}),$ and the conclusion follows by applying Proposition~\ref{proposition1}.
\qed
\end{proof}

It is worth mentioning that necessary and sufficient conditions for the existence of efficient solutions to a multi-objective programming problem, in which the involved functions are locally Lipschitz, were given in \cite{Kim2018}.

Below, we give an example, which shows that the existence result of efficient solutions in the preceding theorem may go awry if the involved functions are {\it not} convex polynomials.
\begin{example}{\rm
Consider $2$-dimensional multi-objective optimization problem
\begin{equation*}
\min_{(x_1,x_2)\in K} \  \big(f_1(x_1,x_2),f_2(x_1,x_2)\big),
\end{equation*}
where $f_1(x_1,x_2)=f_2(x_1,x_2)=(x_1x_2-1)^2+x_2^2$ are (non-convex) polynomials and let $K:=\R^2.$
Note that the image of $K$ under $f=(f_1, f_2)$ is $f(K)=\{(w_1,w_2)\in\R^2: w_1=w_2>0\}.$
So, we see that for any $(z_1,z_2)\in \R^n,$ the section
\begin{equation*}
f(K)\cap \big(f(z_1,z_2)-\R^2_+\big)=\{(w_1,w_2)\in\R^2: 0<w_1=w_2\leq (z_1z_2-1)^2+z_2^2\}
\end{equation*}
is nonempty and bounded, however, it is clear that there is no efficient solution to this problem.
}\end{example}

\begin{remark}{\rm
It is worth noting that, for a multi-objective polynomial optimization problem with same objectives, if the considered polynomial is bound from below but does not attain its infimum, then we can easily verify that ``there is no efficient solution to this problem"; see, e.g., \cite{Fernando2006,Fernando2014,Pham2019,Ueno2008}.
}
\end{remark}

\section{Representation and finite convergence}\label{sect:4}

Let $f_j\colon \R^n \to \R,$ $j=1,\ldots,p,$ and $g_i\colon \R^n \to \R,$ $i=1,\ldots,m,$ be convex polynomials$,$ consider problem~\eqref{modelsetting} and its associated scalar problem~\eqref{hybridmodel} by hybrid method.
Remember that the feasible set of problem~\eqref{hybridmodel} is as follows,
\begin{equation*}
K_z:=\{x\in\R^n: g_i(x) \leq0, \ i=1,\ldots,m, \ f_j(x)\leq f_j(z), \ j=1,\ldots,p\},
\end{equation*}
which is indeed a (basic closed) convex semi-algebraic set in $\R^n$.
By a basic closed semi-algebraic set in $\R^n,$ we mean a set that is defined by finitely many polynomial equalities and inequalities.

In this section, to deal with problem~\eqref{hybridmodel}, we will provide two kinds of representations of non-negativity of convex polynomials over convex semi-algebraic sets.
In addition, we formulate two kinds of Lasserre-type hierarchies of SDP relaxations for problem~\eqref{hybridmodel} and establish their finite convergence results, respectively.

\subsection{Representations of non-negativity of convex polynomials over convex semi-algebraic sets}

Let $z\in K$ be given.
Then, we define the quadratic module $\mathcal{Q}$ generated by the tuples $-g:=(-g_1,\ldots,-g_m)$ and $-f_z:=\Big(-\big(f_1-f_1(z)\big),\ldots,-\big(f_p-f_p(z)\big)\Big)$  as
{\small
\begin{eqnarray*}
\mathcal{Q}(-g,-f_z):=\left\{ \sigma_0-\sum_{i=1}^m\sigma_ig_i-\sum_{j=1}^p\bar\sigma_j\left(f_j-f_j(z)\right)\colon \sigma_0, \ \sigma_1, \ldots, \sigma_m, \ \bar\sigma_1, \ldots, \bar\sigma_p \in \Sigma[x] \right\}.
\end{eqnarray*}
}
Similarly, we define the following special quadratic module generated by the tuples $-g,$ $-f_z$ and an additional polynomial $-\lambda^Tf_z:=-\big(\lambda^Tf-\lambda^Tf(z)\big)$ as
{\small 
\begin{eqnarray*}
&&\mathcal{M}(-g,-f_z,-\lambda^Tf_z)\\
&:=&\left\{ \sigma_0-\sum_{i=1}^m\mu_ig_i-\sum_{j=1}^p\nu_j\left(f_j-f_j(z)\right)-\sigma(\lambda^Tf_z)\colon \sigma,\sigma_0\in \Sigma[x], \  \mu\in\R^m_+, \ \nu\in\R^p_+\right\}.
\end{eqnarray*}
}
Clearly, the module $\mathcal{M}(-g,-f_z,-\lambda^Tf_z)$ is a subset of the quadratic module $\mathcal{Q}(-g,-f_z).$

In connection with problem \eqref{hybridmodel}, we define the Lagrangian-type function $L_z \colon \R^n\times\R^m_+\times\R^p_+\to\R$ as follows:
\begin{equation}\label{Lagrangian function}
L_z(x, \mu,\nu) = \lambda^Tf(x) +\sum_{i=1}^m\mu_i g_i(x)+\sum_{j=1}^p\nu_j\big(f_j(x)-f_j(z)\big).
\end{equation}
\begin{definition}{\rm
We say that the triplet $(\bar x,\bar\mu,\bar \nu)\in K_{z} \times \R^m_+\times\R^p_+$ is a \emph{saddle point} of the Lagrangian-type function $L_z$ defined in \eqref{Lagrangian function}, if the following inequality holds:
\begin{equation*}
L_z(x,\bar \mu,\bar \nu)\geq L_z(\bar x,\bar \mu,\bar \nu) \geq L_z(\bar x,\mu,\nu), \  \forall x \in \R^n, \  \mu \in \R^m_+, \ \nu \in \R^p_+.
\end{equation*}}
\end{definition}

The following lemma, which plays a key role in deriving the desired results, shows that a convex polynomial with positive definiteness of its Hessian at some point is strictly convex and coercive.
\begin{lemma}{\rm(see \cite[Lemma~3.1]{Jeyakumar2014})}\label{lemma2}
Let $f_0$ be a convex polynomial on $\R^n.$
If $\nabla^2f_0 (x_0) \succ 0$ at some point $x_0 \in \R^n,$ then $f_0$ is coercive and strictly convex on $\R^n.$
\end{lemma}

Now, we give the first representation result for non-negativity of convex polynomials over convex semi-algebraic sets.
Note that the result can be obtained by modifying the proof of \cite[Theorem~3.1]{Jeyakumar2014}; apart from this, we also need the following assumption.
\begin{assumption}\label{assumption1}
Let $z_0\in K.$
There exists a saddle point $(\bar x, \bar \mu,\bar\nu) \in K_{z_0} \times \R^m_+\times \R^p_+$ of the Lagrangian-type function $L_{z_0}$ such that $\nabla^2_{xx}L_{z_0}(\bar x, \bar\mu, \bar\nu)\succ 0.$
\end{assumption}

\begin{theorem}{\rm (cf. \cite[Theorem~3.1]{Jeyakumar2014})}\label{theorem2}
Let $f_j\colon \R^n \to \R,$ $j=1,\ldots,p,$ and $g_i\colon \R^n \to \R,$ $i=1,\ldots,m,$ be convex polynomials.
Consider  problem \eqref{hybridmodel} at $z=z_0\in K,$ where $K$ is defined as \eqref{feasibleset}.
If Assumption~\ref{assumption1} holds$,$ then
\begin{equation*}
\lambda^Tf - \bar f_{z_0} \in \mathcal{Q}(-g,-f_{z_0}).
\end{equation*}
\end{theorem}

\begin{proof}
Since $(\bar x, \bar \mu,\bar\nu)$ is a saddle-point of the Lagrangian-type function $L_{z_0},$
it follows that
\begin{equation*}
L_{z_0}(x,\bar \mu,\bar \nu)\geq L_{z_0}(\bar x,\bar \mu,\bar \nu)=\lambda^Tf(\bar x), \  \ \forall x\in\R^n
\end{equation*}
and $\bar x$ is an optimal solution to  problem $({\rm P}_{z_0}).$
Define a function $h\colon\R^n\to\R$ by
\begin{align*}
h(x) :=&\ L_{z_0}(x,\bar \mu,\bar \nu) -\lambda^Tf(\bar x) \\
=&\ \lambda^Tf(x) +\sum_{i=1}^m\bar\mu_ig_i(x)+\sum_{j=1}^p\bar\nu_j\big(f_j(x)-f_j(z_0)\big)-\lambda^Tf(\bar x).
\end{align*}
Then $h$ is a convex polynomial and $h(x) \geq 0$ for all $x \in \R^n.$
Moreover, we easily see that $h(\bar x) = 0 = \inf_{x\in\R^n} h(x).$
Since $\nabla^2_{xx}L_{z_0}(\bar x, \bar\mu, \bar\nu)\succ 0,$ we also see that the Hessian $\nabla^2h(\bar x)$ is positive definite.
It follows from Lemma~\ref{lemma2} that the convex polynomial $h$ is strictly convex and coercive.
Furthermore, this implies that $\bar x$ is the unique optimal solution to $h$ over $\R^n.$
Now consider the set
\begin{equation*}
S :=\{x \in \R^n : c-h(x) +\lambda^Tf(z_0) - \lambda^Tf(\bar x)\geq0\},
\end{equation*}
where $c$ is some positive constant.
Since $z_0\in K_{z_0},$ we see that $\bar x\in S,$ and so, the set $S$ is nonempty and compact (since the polynomial $h$ is coercive).
Moreover, since
\begin{align*}
&\ c-h(x) +\lambda^Tf(z_0) - \lambda^Tf(\bar x)\\
=&\ c-\sum_{j=1}^p\lambda_j\big(f_j(x)-f_j(z_0)\big) -\sum_{i=1}^m\bar\mu_ig_i(x)-\sum_{j=1}^p\bar\nu_j\big(f_j(x)-f_j(z_0)\big)\\
=&\ c-\sum_{i=1}^m\bar\mu_ig_i(x)-\sum_{j=1}^p(\lambda_j+\bar\nu_j)\big(f_j(x)-f_j(z_0)\big)\ \in \ \mathcal{Q}(-g,-f_{z_0})
\end{align*}
and $S = \{x \in \R^n : c-h(x) +\lambda^Tf(z_0) - \lambda^Tf(\bar x)\geq0\}$ is compact, the quadratic module $\mathcal{Q}(-g,-f_{z_0})$ is Archimedean.
It follows from \cite[Corollary~3.6]{Scheiderer2005} (see also \cite[Example 3.18]{Scheiderer2003}) that there exist sum-of-squares polynomials $\sigma_0, \sigma_1 \in \Sigma[x]$ such that, for each $x \in \R^n,$
$h(x) = \sigma_0(x) + \sigma_1(x) \big(c-h(x) + \lambda^Tf(z_0) - \lambda^Tf(\bar x)\big).$
Thus, we have
\begin{equation*}
\lambda^Tf-\bar f_{z_0}=\sigma_0+c\sigma_1-\sum_{i=1}^m\big(\bar\mu_i+\bar\mu_i\sigma_1\big)g_i-\sum_{j=1}^p\big(\bar\nu_j+(\lambda_j+\bar\nu_j)\sigma_1\big)\big(f_j-f_j(z_0)\big),
\end{equation*}
thereby establishing the desired result.
\qed
\end{proof}

\begin{assumption}\label{assumption2}
Let $f_j\colon \R^n \to \R,$ $j=1,\ldots,p,$ and $g_i\colon \R^n \to \R,$ $i=1,\ldots,m,$ be convex polynomials.
For a given point $z_0\in K,$ where $K$ is defined as \eqref{feasibleset}$,$ the following two statements hold$:$
\begin{enumerate}
\item[{\rm (i)}] the Slater-type condition holds for problem $({\rm P}_{z_0}),$ that is$,$ there exists $\hat x \in \R^n$ such that $g_i(\hat x) < 0,$ for $i = 1,\ldots,m,$  and $f_j(\hat x)<f_j(z_0),$ $j=1,\ldots,p;$
\item[{\rm (ii)}] $\sum_{j=1}^p\lambda_j\nabla^2f_j(\bar x) \succ 0,$ where $\bar x \in \argmin_{x\in K_{z_0}} \lambda^Tf(x).$
\end{enumerate}
\end{assumption}

In what follows, slightly modifying \cite[Theorem~3.2]{Jeyakumar2014}, we obtain the second representation for non-negativity of convex polynomials over convex semi-algebraic sets, which is sharper than the result of Theorem~\ref{theorem2}.
\begin{theorem}\label{theorem3}{\rm(cf. \cite[Theorem~3.2]{Jeyakumar2014})}
Let $f_j\colon \R^n \to \R,$ $j=1,\ldots,p,$ and $g_i\colon \R^n \to \R,$ $i=1,\ldots,m,$ be convex polynomials.
Consider  problem \eqref{hybridmodel} at $z=z_0\in K,$ where $K$ is defined as \eqref{example1}.
Suppose that Assumption~\ref{assumption2} holds.
Then
\begin{equation*}
\lambda^Tf - \bar f_{z_0} \in \mathcal{M}(-g,-f_{z_0},-\lambda^Tf_{z_0}).
\end{equation*}
\end{theorem}

\begin{proof}
Let $\bar x\in \argmin_{x\in K_{z_0}} \lambda^Tf(x).$
Since the Slater-type condition holds for problem $({\rm P}_{z_0}),$ by the KKT optimality conditions for convex optimization problems, there exist the Lagrangian multipliers $\bar\mu\in\R^m_+$ and $\bar\nu\in\R^p_+$ such that
\begin{align*}
0 =&\ \sum_{j=1}^p\lambda_j\nabla f_j(\bar x)+\sum_{i=1}^m\bar\mu_i\nabla g_i(\bar x)+\sum_{j=1}^p\bar\nu_j\nabla f_j(\bar x),\\
0 =&\ \bar\mu_i g_i(\bar x), \ i=1,\ldots,m,\\
0 =&\ \bar\nu_j \big(f_j(\bar x)-f_j(z_0)\big), \ j=1,\ldots,p.
\end{align*}
By defining a convex polynomial $h\colon \R^n\to\R$ as
\begin{eqnarray*}
h(x):=\lambda^Tf(x)+\sum_{i=1}^m\bar\mu_ig_i(x)+\sum_{j=1}^p\bar\nu_j\big(f_j(x)-f_j(z_0)\big)-\bar f_{z_0}.
\end{eqnarray*}
It is easily verified that $h(x)\geq0$ for all $x\in \R^n,$ particularly, $h(\bar x)=0.$
On the other hands, since $\sum_{j=1}^p\lambda_j\nabla^2f_j(\bar x) \succ 0,$ it follows from Lemma~\ref{lemma2} that $\lambda^Tf$ is strictly convex and coercive on $\R^n.$ Consequently, the set $F:=\{x\in \R^n:c+\lambda^Tf(z_0)-\lambda^Tf(x)\geq 0\}$ is nonempty and compact, where $c$ is some positive constant.
Moreover, the quadratic module $\mathcal{Q}(c-\lambda^Tf_{z_0})$ is Archimedean along with the fact that $c-\lambda^Tf_{z_0}\in\mathcal{Q}(c-\lambda^Tf_{z_0})$ and $F$ is compact.
In addition, as $h(x)\geq0$ for all $x\in F,$ $h(\bar x)=0$ and $\nabla^2h(\bar x)\succ0,$ $\bar x$ is a unique optimal solution to the problem $\min_{x\in F}h(x).$
Thanks to \cite[Corollary~3.6]{Scheiderer2005} (see also \cite[Example 3.18]{Scheiderer2003}), there exist $\sigma, \sigma_0\in\Sigma[x]$ such that $h=\sigma_0+\sigma(c-\lambda^Tf_{z_0}),$ and hence,
\begin{equation*}
\lambda^Tf-\bar f_{z_0}=\sigma_0+c\sigma-\sum_{i=1}^m\bar\mu_ig_i-\sum_{j=1}^p\bar\nu_j\left(f_j-f_j(z_0)\right)-\sigma(\lambda^Tf_{z_0}),
\end{equation*}
which is the desired result.
\qed
\end{proof}

\begin{remark}
{\rm Note that, in Theorem~\ref{theorem3}, the condition $\sum_{j=1}^p\lambda_j\nabla^2f_j(\bar x) \succ 0$ guarantees the compactness of the feasible set $K_{z_0}$ for the problem $({\rm P}_{z_0}).$
Indeed, let $\{x_k\}\subset K_{z_0}$ be an arbitrary sequence.
Then, for each $k\in \N,$ $x_k\in K_\lambda:=\{x\in K:\lambda^Tf(x)\leq\lambda^Tf(z_0)\},$ which is compact as $\lambda^Tf$ is coercive on $K.$
So, there exists a subsequence $\{x_{k_l}\}\subset K_{\lambda}$ such that $x_{k_l}\to x^\ast\in K_\lambda$ as $l\to + \infty$.
Since $x_{k_l}\in K_{z_0}$ for all $l\in \N,$ by the continuity of each $f_j,$ we have $x^\ast\in K_{z_0},$ and so, the set $K_{z_0}$ is nonempty and compact.
This yields that the efficient solution set of \eqref{modelsetting} is nonempty.
Note also that the problem \eqref{modelsetting} admits an efficient solution if there exists $x\in \R^n$ such that $\sum_{j=1}^p\lambda_j\nabla^2f_j(x) \succ 0.$
}\end{remark}

\subsection{Finite convergence for the Lasserre-type hierarchies of semidefinite programming relaxations}
{\leqnomode{
Let $z\in K,$ where $K$ is defined as \eqref{feasibleset}, be given.
With $g_0=1,$ let $r_i:=\lceil\deg g_i/2\rceil,$ $i=0,1,\ldots,m,$ and let $d_j:=\lceil\deg f_j/2\rceil,$ $j=1,\ldots,p,$
where the notation $\lceil a\rceil$ stands for the smallest integer greater than or equal to $a.$
Now, for $k\ge k_0:=\max\{\max_ir_i,\max_jd_j\},$ consider the following semidefinite programming problem:
\begin{align}\label{Qkmodel}
&\rho_z^k:= \inf\limits_{y} \ \ \ \sum_{j=1}^p\sum_{\alpha\in\N^n_{2k}}\lambda_j(f_j)_\alpha y_\alpha  \tag{${\rm Q}^k_z$} \\
& \qquad\ \ {\rm s.t. } \ \ \  M_{k-r_i}(-g_iy)\succeq0, \ i=0,1,\ldots,m, \nonumber \\
& \qquad \qquad \quad M_{k-d_j}\Big(\big(f_j(z)-f_j\big)y\Big)\succeq0, \ j=1,\ldots,p. \nonumber
\end{align}
It is worth noting that  \eqref{Qkmodel} is a Lasserre-type hierarchy of SDP relaxation of problem \eqref{hybridmodel} i.e., $\rho_z^k\leq \rho_z^{k+1}\leq\cdots\leq \bar f_z$ for all  $k\geq k_0$ (see, e.g., \cite{Lasserre2015}).

Now, consider the following programming problem:
\begin{align}\label{DualQkmodel}
&\bar\rho_z^k:= \sup\limits_{\gamma,\sigma,\sigma_i,\bar\sigma_j} \ \ \   \gamma  \tag{$\widehat{\rm Q}^k_z$} \\
& \qquad\quad \ \; {\rm s.t. } \quad \ \ \lambda^Tf-\gamma=\sigma_0-\sum_{i=1}^m\sigma_ig_i-\sum_{j=1}^p\bar\sigma_j\big(f_j-f_j(z)\big), \nonumber \\
& \qquad \qquad \qquad \ \  \sigma_i\in \Sigma[x]_{k-r_i}, \ i=0,1,\ldots, m, \nonumber \\
& \qquad \qquad \qquad \ \  \bar\sigma_j\in \Sigma[x]_{k-d_j}, \ j=1,\ldots,p. \nonumber
\end{align}
It is worth mentioning that  \eqref{DualQkmodel} is the dual problem of  \eqref{Qkmodel} (see, \cite{Lasserre2009,Lasserre2015}).
Note that, for a given $z\in K,$ the set $K_z$ is nonempty.
This implies that the feasible set of  \eqref{Qkmodel} is nonempty.
So, if the feasible set of  \eqref{DualQkmodel} is nonempty, then we see that $\rho_z^k\geq \bar\rho_z^k$ for all $k\geq k_0$ by weak duality.
Moreover, \eqref{DualQkmodel} has an asymptotic convergence in the sense that $\bar\rho_z^k\uparrow\bar f_z$ as $k\to\infty$ without any regularity conditions (see, e.g., \cite[Theorem~2.1]{Jeyakumar2014}).
}}

\medskip
Now, with the help of Theorem~\ref{theorem2}, we show the finite convergence for the hierarchy of SDP relaxations of \eqref{hybridmodel} in the next Theorem.
\begin{theorem}\label{theorem4}
Let $f_j\colon \R^n \to \R,$ $j=1,\ldots,p,$ and $g_i\colon \R^n \to \R,$ $i=1,\ldots,m,$ be convex polynomials.
Consider problem \eqref{hybridmodel} at $z=z_0\in K,$ where $K$ is defined as \eqref{feasibleset}.
If Assumption~\ref{assumption1} is satisfied$,$ then there exists an integer $\bar k$ such that $\bar \rho_{z_0}^k=\rho_{z_0}^k=\bar f_{z_0}$ for all $k\geq \bar k,$ and both  $(\widehat{\rm Q}^k_{z_0})$ and  $({\rm Q}^k_{z_0})$ attain their optimal solutions.
\end{theorem}

\begin{proof}
Thanks to Theorem~\ref{theorem2}, there exist sum-of-squares polynomials $\sigma_i,$ $i=0,1,\ldots,m,$ and $\bar\sigma_j,$ $j=1,\ldots,p,$  such that
\begin{equation*}
\lambda^Tf - \bar f_{z_0} =\sigma_0-\sum_{i=1}^m\sigma_ig_i-\sum_{j=1}^p\bar\sigma_j\big(f_j-f_j(z_0)\big).
\end{equation*}
Now, let $\bar k\geq\max\left\{\deg\sigma_0,\max_i\{\deg(\sigma_ig_i)\},\max_j\{\deg(\bar\sigma_jf_j)\}\right\}.$
Then, we have $\big(\bar f_{z_0},\sigma_0,(\sigma_i),(\bar\sigma_j)\big)$ is a feasible solution of  $(\widehat{\rm Q}^k_{z_0})$ for $k=\bar k,$
and so, $\bar f_{z_0}\leq \bar \rho^{k}_{z_0}$ for $k\geq \bar k.$
Also, we can easily see that the sequence $\{\bar{\rho}_{z_0}^{k}\}$ is monotonically increasing and bounded from above by $\bar{f}_{z_0},$ in particular, $\bar{\rho}_{z_0}^{k}\leq\bar{f}_{z_0}$ for all $k\geq\bar k.$
Thus, $\bar{\rho}_{z_0}^{k}=\bar{f}_{z_0}$ for all $k\geq \bar k.$ In fact, $\big(\bar f_{z_0},\sigma_0,(\sigma_i),(\bar\sigma_j)\big)$ is an optimal solution of  $(\widehat{\rm Q}^k_{z_0})$ for all $k\geq \bar k.$

On the other hand, by the weak duality between \eqref{Qkmodel} and \eqref{DualQkmodel}, we have $\bar \rho_{z_0}^k\leq\rho_{z_0}^k$ for all $k\geq \bar k.$
Thus, we conclude that $\bar \rho_{z_0}^k= \rho_{z_0}^k=\bar f_{z_0}$ for all $k\geq \bar k.$
In particular, it is clear that, for all $k\geq \bar k,$ $\bar y=v_{2k}(\bar x)$ is an optimal solution to  $({\rm Q}^k_{z_0}),$ which completes the proof.
\qed
\end{proof}

\subsection{Finite convergence for the Lasserre-type hierarchy of sharp semidefinite programming relaxations}
{\leqnomode{
Let $z\in K,$ {where $K$ is defined as \eqref{feasibleset},} be given.
Consider the following semidefine programming problem:
\begin{align}\label{Pkmodel}
&f_z^k:= \inf\limits_{y} \ \ \ \sum_{j=1}^p\sum_{\alpha\in\N^n_{2k}}\lambda_j(f_j)_\alpha y_\alpha  \tag{${\rm P}^k_z$} \\
& \qquad\ \ {\rm s.t. } \ \ \  M_k(y)\succeq0, \nonumber \\
& \qquad \qquad \quad \sum_{\alpha\in\N^n_{2k}}(g_i)_\alpha y_\alpha\leq0, \  i=1,\ldots,m, \nonumber \\
& \qquad \qquad \quad \sum_{\alpha\in\N^n_{2k}}(f_j)_\alpha y_\alpha\leq f_j(z), \ j=1,\ldots,p, \nonumber \\
& \qquad \qquad \quad M_{k-d_f}\big((-\lambda^Tf_{z})y\big)\succeq0, \nonumber
\end{align}
where $k\ge k_0=\max\{\max_jd_j,\max_ir_i\}$ and $d_f:=\max_jd_j.$
It is worth noting that  \eqref{Pkmodel} is also Lasserre-type hierarchy of SDP relaxation for \eqref{hybridmodel}.
Indeed, letting $x\in K_z$ and  $y:=v_{2k}(x),$ we see that $y$ is a feasible solution of \eqref{Pkmodel} with its value $f(x).$
So, we have $f_z^k\le \bar f_z$ for all $k\ge k_0.$
In addition, we see that $f_z^k\le f_z^{k+1}$ for all $k\ge k_0$ since  $({\rm P}^{k+1}_z)$ is more constrained than  $({\rm P}^{k}_z).$
}}

Below, we consider the dual problem \eqref{DualPkmodel} of  \eqref{Pkmodel} as follows:
 \begin{align}\label{DualPkmodel}
&\bar f_z^k:= \sup\limits_{\gamma,\sigma,\sigma_0,\mu,\nu} \ \ \   \gamma  \tag{$\widehat{\rm P}^k_z$} \\
& \qquad\quad \ \; {\rm s.t. } \quad \ \ \lambda^Tf-\gamma=\sigma_0-\sum_{i=1}^m\mu_ig_i-\sum_{j=1}^p\nu_j\big(f_j-f_j(z)\big) - \sigma (\lambda^Tf_{z} ), \nonumber \\
& \qquad \qquad \qquad \ \  \sigma\in\Sigma[x]_{k}, \ \sigma_0\in\Sigma[x]_{k-d_f}, \nonumber \\
& \qquad \qquad \qquad \ \  \mu_i\geq0,\ i =1,\ldots,m, \ \nu_j\geq0, \ j=1,\ldots,p. \nonumber
\end{align}
Note that weak duality holds between \eqref{Pkmodel} and \eqref{DualPkmodel}, i.e., $\bar f_z^k\leq f_z^k$ for all $k\geq k_0.$
Moreover, it is easily to verify that $\bar f_z^k\leq \bar f_z^{k+1}$ for all $k\geq k_0.$

We now establish an asymptotic convergence result for the SDP relaxations \eqref{DualPkmodel} under the positive definiteness of the Hessian of the objective function of problem \eqref{hybridmodel} at some point, and the proof of this result can be obtained by slightly modifying the proof of \cite[Theorem~2.1]{Jeyakumar2014}.
\begin{theorem}
Let $f_j\colon \R^n \to \R,$ $j=1,\ldots,p,$ and $g_i\colon \R^n \to \R,$ $i=1,\ldots,m,$ be convex polynomials.
Let $z_0\in K,$ where $K$ is defined as \eqref{feasibleset}$,$ be given.
If there exists $\tilde x\in\R^n$ such that $\nabla^2(\lambda^Tf)(\tilde x)\succ0,$ then $\bar f_{z_0}^k\uparrow\bar f_{z_0}$ as $k\to\infty.$
\end{theorem}

\begin{proof}
Let $\epsilon > 0.$
We first claim that there exist  $\mu \in \R^m_+$ and $\nu\in\R^p_+$ such that
\begin{equation}\label{theorem5rel1}
\lambda^Tf(x) - \bar f_{z_0} +\sum_{i=1}^m\mu_ig_i(x)+\sum_{j=1}^p\nu_j\big(f_j(x)-f_j(z_0)\big)+\epsilon > 0, \  \ \forall x \in \R^n.
\end{equation}
Since  $\lambda^Tf(x) - \bar f_{z_0}\geq0$ for all $x\in K_{z_0},$  observe that $\lambda^Tf - \bar f_{z_0} + \epsilon  > 0$ on $K_{z_0}.$
Then there exists $\delta > 0$ such that $\lambda^Tf(x) - \bar f_{z_0} + \epsilon > 0$ for all $x\in K_{z_0,\delta},$ where
\begin{equation*}
K_{z_0,\delta}:=\{x\in\R^n:g_i(x)\leq\delta, \ i=1,\ldots,m, \ f_j(x)-f_j(z_0)\leq\delta, \ j=1,\ldots,p\}.
\end{equation*}
[Otherwise,  suppose that there exist sequences $\{\delta_k\} \subset \R_+,$ $\delta_k \to 0$ and $\{x_k\} \subset \R^n$ such that
$g_i(x_k) \leq \delta_k,$ $i = 1,\ldots,m,$ $f_j(x_k)-f_j(z_0)\leq\delta_k,$ $j=1,\ldots,p,$ and $\lambda^Tf(x_k)-\bar f_{z_0} + \epsilon \leq 0.$
Then,
\begin{eqnarray*}
0 &\leq& \inf_{x,w}\bigg\{\sum_{i=1}^{m+p}w_i^2 : \lambda^Tf(x)-\bar f_{z_0} + \epsilon \leq 0, \ g_i(x)  \leq  w_i, \ i = 1,\ldots ,m,  \\
  &&\hspace{23mm}f_j(x)-f_j(z_0)\leq w_{m+j}, \ j=1,\ldots,p\bigg\}\\
&\leq&\sum_{i=1}^{m+p}\delta_k^2 = (m+p)\delta_k^2 \to 0 \textrm{ as } k \to \infty.
\end{eqnarray*}
It follows from Lemma~\ref{lemma1} that there exist $x^\ast\in\R^n$ and $w^\ast\in\R^{m+p}$ such that
$\lambda^Tf(x^\ast)-\bar f_{z_0} + \epsilon \leq 0,$ $g_i(x^\ast)  \leq  w^\ast_i,$ $i = 1,\ldots ,m,$ $f_j(x^\ast)-f_j(z_0)\leq w^\ast_{j+m},$ $j=1,\ldots,p,$ and $\sum_{i=1}^{m+p}(w^\ast_i)^2=0,$
i.e., $\lambda^Tf(x^\ast)-\bar f_{z_0} + \epsilon \leq 0,$ $g_i(x^\ast)  \leq  0,$ $i = 1,\ldots ,m,$ $f_j(x^\ast)-f_j(z_0)\leq 0,$ $j=1,\ldots,p,$
which contradicts the fact that $\lambda^Tf - \bar f_{z_0} + \epsilon$ is positive over $K_{z_0}.$]

Now, define $h\colon\R^n\to\R$ by
\begin{equation*}
h(x):=\lambda^Tf(x) - \bar f_{z_0} +\sum_{i=1}^m\mu_ig_i(x)+\sum_{j=1}^p\nu_j\big(f_j(x)-f_j(z_0)\big)+\epsilon, \  \forall x \in \R^n.
\end{equation*}
Along with \eqref{theorem5rel1}, it is clear that $h(x)>0$ for all $x\in\R^n.$

On the other hand, since $\sum_{j=1}^p\lambda_j\nabla^2f_j(\tilde x)\succ0$ for some $\tilde x\in\R^n,$ it follows from Lemma~\ref{lemma2} that $\lambda^Tf$ is strictly convex and coercive on $\R^n.$
Let $S:=\{x\in\R^n:-\lambda^Tf_{z_0}(x)\geq0\}.$
Then the set $S$ is nonempty and compact.
Note that $h$ is positive on $S.$
Moreover, since $-\lambda^Tf_{z_0} \in \mathcal{Q}(-\lambda^Tf_{z_0})$ and $S$ is compact, the quadratic module $\mathcal{Q}(-\lambda^Tf_{z_0})$ is Archimedean.
Thanks to Lemma~\ref{Putinar} (Putinar's Positivstellensatz), there exist $\sigma,\sigma_0\in\Sigma[x]$ such that $h=\sigma_0-\sigma(\lambda^Tf_{z_0}),$
i.e., for each $x\in\R^n,$
\begin{eqnarray*}
\lambda^Tf(x) - \bar f_{z_0} +\epsilon=\sigma_0-\sum_{i=1}^m\mu_ig_i(x)-\sum_{j=1}^p\nu_j\big(f_j(x)-f_j(z_0)\big)-\sigma(\lambda^Tf_{z_0}).
\end{eqnarray*}
So, $(\bar f_{z_0}-\epsilon,\sigma,\sigma_0,\mu,\nu)$ is a feasible solution of $(\widehat{\rm P}^k_{z_0})$ as soon as $k$ is large enough.
Hence we have $\bar f_{z_0}-\epsilon\leq \bar f_{z_0}^k.$
Finally, by weak duality between $({\rm P}^k_{z_0})$ and $(\widehat{\rm P}^k_{z_0}),$ we have $\bar f_{z_0}^k\leq f_{z_0}^k$ for all $k\geq k_0.$
Besides, as shown before that $f_{z_0}^k\leq \bar f_{z_0},$ we thus conclude that $\bar f_{z_0}-\epsilon\leq \bar f_{z_0}^k\leq \bar f_{z_0}.$
As $\epsilon>0$ is arbitrary, the desired result follows.
\qed
\end{proof}

We close this section by giving the next finite convergence result for the hierarchy of SDP relaxations of  \eqref{hybridmodel}, which is  sharper than the one of Theorem~\ref{theorem4}.
\begin{theorem}\label{theorem6}
Let $f_j\colon \R^n \to \R,$ $j=1,\ldots,p,$ and $g_i\colon \R^n \to \R,$ $i=1,\ldots,m,$ be convex polynomials.
Consider problem~\eqref{hybridmodel} at $z=z_0\in K,$ where $K$ is defined as \eqref{feasibleset}.
If Assumption~\ref{assumption2} holds$,$ then there exists an integer $\bar k$ such that $\bar f_{z_0}^k=f_{z_0}^k=\bar f_{z_0}$ for all $k\geq \bar k.$
In addition$,$ both $({\rm P}^k_{z_0})$ and $(\widehat{\rm P}^k_{z_0})$ attain their optimal solutions.
\end{theorem}

\begin{proof}
The proof is similar to the one of Theorem~\ref{theorem4}.
It follows from Theorem~\ref{theorem3} that there exist $\sigma,\sigma_0\in\Sigma[x],$ $\bar\mu\in\R^m_+,$ and $\bar\nu\in\R^p_+$ such that
\begin{equation*}
\lambda^Tf-\bar f_{z_0}=\sigma_0-\sum_{i=1}^m\bar\mu_ig_i-\sum_{j=1}^p\bar\nu_j\big(f_j-f_j(z_0)\big)-\sigma(\lambda^Tf_{z_0}).
\end{equation*}
Let $\bar k\geq\max\{\deg\sigma_0,\deg\sigma+d_f,\max_i\{\deg g_i\}\}.$
Then, $(\bar f_{z_0},\sigma,\sigma_0,\bar\mu,\bar\nu)$ is a feasible solution of  $(\widehat{\rm P}^k_{z_0})$ for $k=\bar k,$
and so, we have $\bar f_{z_0}\leq \bar f^{\bar k}_{z_0}.$
Moreover, we have already seen that $\bar{f}_{z_0}^{k}\leq\bar{f}_{z_0}$ for all $k\geq\bar k.$
Thus, we have $\bar{f}_{z_0}^{k}=\bar{f}_{z_0}$ for all $k\geq \bar k.$ In fact, $(\bar f_{z_0},\sigma,\sigma_0,\bar\mu,\bar\nu)$ is an optimal solution to  $(\widehat{\rm P}^k_{z_0})$ for all $k\geq \bar k.$

On the other hand, by the weak duality between $({\rm P}^k_{z_0})$ and $(\widehat{\rm P}^k_{z_0}),$ $\bar f_{z_0}^k\leq f_{z_0}^k$ for all $k\geq k_0.$
Thus, we conclude that $\bar f_{z_0}^k=\ f_{z_0}^k=\bar f_{z_0}$ for all $k\geq \bar k.$
In particular, it is clear that, for all $k\geq \bar k,$ $\bar y=v_{2k}(\bar x)$ is an optimal solution to  $({\rm P}^k_{z_0}),$ which completes the proof.
\qed
\end{proof}

\section{Finding efficient solutions}\label{sect:5}
Let $f_j\colon \R^n \to \R,$ $j=1,\ldots,p,$ and $g_i\colon \R^n \to \R,$ $i=1,\ldots,m,$ be convex polynomials.
Let $z\in K,$ where $K$ is defined as \eqref{feasibleset}, be given, and let $\bar y$ be an optimal solution to \eqref{Qkmodel} (or \eqref{Pkmodel}).
Moreover, let $r_i:=\lceil\deg g_i/2\rceil,$ $i=0,1,\ldots,m,$ and let $d_j:=\lceil\deg f_j/2\rceil,$ $j=1,\ldots,p,$
where the notation $\lceil a\rceil$ stands for the smallest integer greater than or equal to $a.$ Denote also $k_0:=\max\{\max_ir_i,\max_jd_j\}.$
If the flat extension condition holds, that is,
\begin{equation}\label{flat extension}
{\rm rank\,}M_k(\bar y)={\rm rank\,}M_{k-k_0}(\bar y),
\end{equation}
then there exist at least ${\rm rank\,}M_k(\bar y)$ optimal solutions to \eqref{hybridmodel} (see, e.g., \cite[Theorem~6.6]{Lasserre2015}), and they can be efficiently extracted by a suitable algorithm (see, e.g., \cite[Algorithm~6.9 in Section 6.1]{Lasserre2015}, \cite[Section~2]{Henrion2005}).
In addition, the flat extension condition \eqref{flat extension} guarantees a finite convergence of Lasserre's hierarchy, but the {\it converse} may not be true \cite[Example~1.1]{Nie2013}.

Recently, a weak condition of the flat extension condition \eqref{flat extension} was proposed by Nie \cite{Nie2013}. That is, there exists an integer $t\in [k_0, k]$ such that
\begin{equation}\label{flat truncation}
{\rm rank\,}M_t(\bar y)={\rm rank\,}M_{t-k_0}(\bar y).
\end{equation}
Also, we say that $\bar y$ has a \emph{flat truncation} if the condition \eqref{flat truncation} holds for some $t\in [k_0, k].$
Note that if $\bar y$ has a flat truncation, then we can find at least ${\rm rank\,}M_t(\bar y)$ optimal solutions to \eqref{hybridmodel}.
\begin{assumption}{\rm (cf. Assumption~2.1 in \cite{Nie2013})}\label{assumption3}
Let $\lambda\in{\rm int\,}\R^p_+$ be fixed.
For a given $z\in K,$ where $K$ is defined as \eqref{feasibleset}$,$
 there exists $\varrho\in \mathcal{Q}(-g,-f_z)$ such that for every $I\subseteq\{1,\ldots,m\},$ $J\subseteq\{1,\ldots,p\},$ and
\begin{eqnarray*}
\mathcal{V}_{z,I,J}:=\{x\in\R^n \hspace{-1mm}&\colon&\hspace{-1mm} \textrm{there exist } \mu_i\geq0, \ i\in I, \textrm{ and } \nu_j\geq0, \ j\in J, \textrm{ such that } \\
&&\hspace{-1mm} \sum_{j=1}^p\lambda_j\nabla f_j(x)+\sum_{i\in I}\mu_i\nabla g_i(x)+\sum_{j\in J}\nu_j\nabla f_j(x)=0, \\
&&\hspace{-1mm}g_i(x)=0 \ (\forall i\in I), \ f_j(x)-f_j(z)=0 \ (\forall j\in J)\},
\end{eqnarray*}
the intersection $\mathcal{V}_{z,I,J}\cap \mathcal{S}_z\cap\mathcal{P}$ is finite, where
$\mathcal{S}_z:=\{x\in\R^n:\lambda^Tf(x)=\bar f_z\},$ and $\mathcal{P}:=\{x\in\R^n: \varrho(x)\geq0\}.$
\end{assumption}

\begin{remark}
{\rm
\begin{itemize}
\item[{\rm (i)}] It is worth noting that Assumption~\ref{assumption3} holds generically, which means if the polynomials $f_j,$ $j=1,\ldots,p,$ and $g_i,$ $i=1,\ldots,m,$ are generic (not necessarily convex), then for the given $z \in K,$ the set $\mathcal{V}_{z,I,J}$ is finite for every index set $I\subseteq\{1,\ldots,m\}$ and $J\subseteq\{1,\ldots,p\},$ and so Assumption~\ref{assumption3} holds by choosing $\varrho=0$ (for more details, see \cite[Proposition~2.1]{Nie2009}).
\item[{\rm (ii)}] It is also worth emphasizing that under Assumption~\ref{assumption3}, problem~\eqref{hybridmodel} generically admits a unique optimal solution for any fixed $\lambda\in{\rm int\,}\R^p_+.$ Indeed, problem~\eqref{hybridmodel} is a convex polynomial optimization problem.
    Also, Assumption~\ref{assumption3} implies that  $({\rm P}_{z_0})$ has finite optimal solutions (see \cite{Nie2013}). This, combining with the convexity assumption, implies the uniqueness of the optimal solution to problem~\eqref{hybridmodel}.
\end{itemize}
}
\end{remark}

In fact, Assumption~\ref{assumption3} guarantees that the flat truncation \eqref{flat truncation} is not only a sufficient condition, but also a {\it necessary} condition for the finite convergence of the Lasserre hierarchy \cite[Theorem~2.2]{Nie2013}.

\medskip
Along with these facts, we propose the following result.
\begin{theorem}\label{theorem7}
Let $f_j\colon \R^n \to \R,$ $j=1,\ldots,p,$ and $g_i\colon \R^n \to \R,$ $i=1,\ldots,m,$ be convex polynomials$,$ consider problem~\eqref{hybridmodel} at $z=z_0\in K,$ where $K$ is defined as \eqref{feasibleset}.
If Assumptions~\ref{assumption1} and \ref{assumption3} hold$,$ then $\bar x:=(\bar y_\alpha)_{|\alpha|=1}$ is an efficient solution to problem $\eqref{modelsetting},$ where $\bar y$ is an optimal solution to $({\rm Q}^k_{z_0})$  for some sufficiently large $k.$
\end{theorem}

\begin{proof}
By Theorem~\ref{theorem4}, for sufficiently large $k,$ $\bar \rho_{z_0}^k=\rho_{z_0}^k=\bar f_{z_0}$ and the optimal value of $(\widehat{\rm Q}^k_{z_0})$ is achievable.
It follows from \cite[Theorem~2.2]{Nie2013} that every optimal solution to $({\rm Q}^k_{z_0})$ has a flat truncation for some sufficiently large $k,$
i.e., there exists an integer $t\in [k_0, k]$ such that
\begin{equation*}
{\rm rank\,}M_t(\bar y)={\rm rank\,}M_{t-k_0}(\bar y),
\end{equation*}
where $\bar y$ is an optimal solution to $({\rm Q}^k_{z_0}),$ and so, the problem $({\rm P}_{z_0})$ has at least ${\rm rank\,}M_t(\bar y)$ optimal solutions.

On the other hand, Assumption~\ref{assumption3} implies that $({\rm P}_{z_0})$ has a unique solution, thus ${\rm rank\,}M_t(\bar y)$ and ${\rm rank\,}M_{t-k_0}(\bar y)$ should be equal to $1,$ and so, necessarily, $M_t(\bar y)=v_t(\bar x)v_t(\bar x)^T$ for some $\bar x\in\R^n.$
Moreover, since $\bar y$ is a feasible solution of $({\rm Q}^k_{z_0}),$  we can easily see that $\bar x$ is also a feasible solution of $({\rm P}_{z_0}).$
It means that $\bar y$ is the vector of moments up to order $2t$ of the Dirac measure $\delta_{\bar x}$ at $\bar x\in K_{z_0},$ i.e., $\bar y=v_{2t}(\bar x).$
This yields that $\bar x$ is an optimal solution to $({\rm P}_{z_0}).$ In particular, $\bar x=(\bar y_\alpha)_{|\alpha|=1}.$
It follows from Proposition~\ref{proposition1} that $(\bar y_\alpha)_{|\alpha|=1}$ is an efficient solution to \eqref{modelsetting}.
\qed
\end{proof}

The following lemma shows that a weak condition of Assumption~\ref{assumption2} (ii) implies the validity of Assumption~\ref{assumption3}.
\begin{lemma}\label{lemma3}
Let $z_0\in K,$ where $K$ is defined as \eqref{feasibleset}$,$ be given.
Assume that there exists $\tilde x\in\R^n$ such that the Hessian $\nabla^2(\lambda^Tf)(\tilde x)$ is positive definite.
Then Assumption~\ref{assumption3} holds.
\end{lemma}
\begin{proof}
Let $z_0\in K$ be fixed.
Since the Hessian $\sum_{j=1}^p\lambda_j\nabla^2f_j(\tilde x)$ is positive definite, it follows from Lemma~\ref{lemma2} that the polynomial $\lambda^Tf$ is coercive and strictly convex.
This implies that there is a unique optimal solution $\bar x$ to the problem $({\rm P}_{z_0})$.

Let us denote the set of active constraints by
\begin{equation*}
I(\bar x)\cup J(\bar x):=\{i:g_i(\bar x)=0\}\cup \{j:f_j(\bar x)-f_j(z_0)=0\}.
\end{equation*}
Then, without loss of generality, we can assume that the set $I(\bar x)\cup J(\bar x)$ is nonempty; otherwise, we have $\nabla \left(\lambda^Tf\right)(\bar x) = 0.$
This implies that $\mathcal{S}_{z_0}=\{x\in\R^n:\lambda^Tf(x)=\bar f_{z_0}\}=\{\bar x\},$ and so, in this case, Assumption~\ref{assumption3} obviously holds.

Now, let $I\subset \{1,\ldots,m\}$ and $J\subset\{1,\ldots,p\}$ be any fixed.
To finish the proof of this lemma, it suffers to show that the following two statements hold:
\begin{itemize}
  \item[\textrm{(i)}] If $\bar x\in \mathcal{V}_{z_0,I,J}\cap\mathcal{S}_{z_0},$ then $\mathcal{V}_{z_0,I,J}\cap\mathcal{S}_{z_0}=\{\bar x\};$
  \item[\textrm{(ii)}] otherwise, $\mathcal{V}_{z_0,I,J}\cap\mathcal{S}_{z_0}=\emptyset.$
\end{itemize}
We first prove that the assertion (i) is true.
Assume to the contrary that there exists $\hat x\in\R^n$ such that $\hat x\neq \bar x$ and $\hat x\in \mathcal{V}_{z_0,I,J}\cap\mathcal{S}_{z_0}.$
Then, there exist $\hat\mu_i\geq0,$ $i\in I,$ and $\hat\nu_j\geq0,$ $j\in J,$ such that
\begin{equation}
\begin{split}\label{lemma3_rel1}
\sum_{j=1}^p\lambda_j\nabla f_j(\hat x)+\sum_{i\in I}\hat\mu_i\nabla g_i(\hat x)+\sum_{j\in J}\hat\nu_j\nabla f_j(\hat x)=0&,\\
g_i(\hat x)=0, \ i\in I,&\\
f_j(\hat x)-f_j(z_0)=0, \ j\in J.&
\end{split}
\end{equation}
Now, consider the convex optimization problem
{\leqnomode{
\begin{align}\label{newmodel}
\min\limits_{x\in\R^n} &\quad \lambda^Tf(x)  \tag{${\rm \widehat{P}}_{z_0}$} \\
 {\rm s.t. } &\quad g_i(x) \leq0, \ i\in I, \nonumber \\
 &\quad f_j(x)\leq f_j(z_0), \ j\in J. \nonumber
\end{align}
Then $\hat x$ is a feasible solution of problem~\eqref{newmodel}
It follows from \eqref{lemma3_rel1} that $\hat x$ is indeed an optimal solution to problem~\eqref{newmodel} with the optimal value $\lambda^Tf(\hat x)=\bar f_{z_0}.$
On the other hand, since the set $K_{z_0}$ is clearly a subset of the feasible set of problem~\eqref{newmodel}, $\bar x$ is a feasible solution to problem~\eqref{newmodel} with the value $\bar f_{z_0},$
and so, $\bar x$ is also an optimal solution of problem~\eqref{newmodel}, which contradicts to the fact that the problem~\eqref{newmodel} has a unique optimal solution (due to the strictly convexity of $\lambda^Tf$).
Thus, the statement (i) holds.
}}

(ii) Suppose to the contrary that the set $\mathcal{V}_{z_0,I,J}\cap\mathcal{S}_{z_0}$ is nonempty.
For simplicity, let $\hat x\in \mathcal{V}_{z_0,I,J}\cap\mathcal{S}_{z_0}.$
Then similar to the proof of assertion (i), we see that $\hat x$ is an optimal solution to problem~\eqref{newmodel},
and so, we arrive at a contradiction to the fact that $\bar x$ is the unique optimal solution to problem~\eqref{newmodel}.

Since $I$ and $J$ are arbitrary, as a consequence, for every $I\subset \{1,\ldots,m\}$ and $J\subset\{1,\ldots,p\},$ the set $\mathcal{V}_{z_0,I,J}\cap\mathcal{S}_{z_0}$ is finite,
and thus, the desired result follows.
\qed
\end{proof}

The following example shows that Lemma~\ref{lemma3} may fail if the Hessian $\nabla^2(\lambda^Tf)(x)$ is not positive definite for all $x.$
\begin{example}{\rm
For simplicity, let us consider the following $2$-dimensional scalar convex polynomial optimization problem:
\begin{eqnarray*}
({\rm P}_z) \quad & \min\limits_{(x_1,x_2)\in\R^2} & f_1(x_1, x_2) \\
& {\rm s.t. } & g_1(x_1,x_2) \leq0, \\
& & f_1(x_1,x_2)\leq f_1(z_1,z_2),
\end{eqnarray*}
where $f_1(x_1,x_2):=(x_1-x_2)^2$ and $g_1(x_1,x_2)=x_1-x_2+1.$
we first note that a simple calculation shows that the Hessian $\nabla^2f_1(x_1,x_2)$ is not positive definite for all $(x_1,x_2)\in\R^2.$
Let $z_0=(0,1)\in \{(x_1,x_2)\in\R^2:g_1(x_1,x_2)\leq0\}.$
Then it is easily verified that $\mathcal{S}_{z_0}=\{(x_1,x_2)\in\R^2:(x_1-x_2)^2=1\}.$
Moreover, for $I=\{1\}$ and $J=\emptyset,$ we have
\begin{eqnarray*}
\mathcal{V}_{z_0,\{1\},\emptyset}&=&\{(x_1,x_2)\in\R^2 : \textrm{there exists } \mu_1\geq0 \textrm{ such that } x_1-x_2+1=0 ,\\
&&\hspace{30mm}\left(\begin{array}{r}
 x_1-x_2 \\
-x_1-x_2 \\
\end{array}\right)
+\mu_1\left(\begin{array}{r}
 1 \\
-1 \\
\end{array}\right)
=\left(\begin{array}{c}
 0 \\
 0 \\
\end{array}\right) \},\\
&=&\{(x_1,x_2)\in\R^2:x_1-x_2+1=0\},
\end{eqnarray*}
and so, we get $\mathcal{V}_{z_0,\{1\},\emptyset}\cap \mathcal{S}_{z_0}=\{(x_1,x_2)\in\R^2:x_1-x_2+1=0\}.$
Let us arbitrary choose a polynomial $\varrho$ in the quadratic module $\mathcal{Q}\big(-g_1,-f_1+f_1(z_0)\big).$
Observe that $\varrho(x)\geq0$ for all $x\in \mathcal{V}_{z_0,\{1\},\emptyset}\cap \mathcal{S}_{z_0},$ i.e.,
\begin{equation*}
\mathcal{V}_{z_0,\{1\},\emptyset}\cap \mathcal{S}_{z_0}\cap \mathcal{P}=\{(x_1,x_2)\in\R^2:x_1-x_2+1=0\},
\end{equation*}
thereby, Assumption~\ref{assumption3} does not hold.
}\end{example}

We are now ready to provide our final result which shows that finding efficient solutions to \eqref{modelsetting} can be done via solving the Lasserre-tpye hierarchy of SDP relaxations.
\begin{theorem}\label{theorem8}
Let $f_j\colon \R^n \to \R,$ $j=1,\ldots,p,$ and $g_i\colon \R^n \to \R,$ $i=1,\ldots,m,$ be convex polynomials$,$ consider problem \eqref{hybridmodel} at $z=z_0\in K,$ where $K$ is defined as \eqref{feasibleset}.
If Assumption~\ref{assumption2} holds$,$ then $\bar x:=(\bar y_\alpha)_{|\alpha|=1}$ is an efficient solution to \eqref{modelsetting}$,$ where $\bar y$ is an optimal solution of $({\rm P}^k_{z_0})$  for some sufficiently large $k.$
\end{theorem}

\begin{proof}
It follows from Theorem~\ref{theorem6} that for sufficiently large $k,$ $\bar f_{z_0}^k=f_{z_0}^k=\bar f_{z_0}$ and the optimal value of $(\widehat{\rm P}^k_{z_0})$ is achievable.
Also, by Lemma~\ref{lemma3}, Assumption~\ref{assumption3} holds, and the rest of the proof of this theorem can be constructed by using similar arguments as in the proof of Theorem~\ref{theorem7}.
\qed
\end{proof}

\begin{remark}{\rm
By employing the hybrid method to \eqref{modelsetting}, Algorithm~\ref{algorithm1} below shows that finding  efficient solutions to \eqref{modelsetting} can be done via solving hierarchies of SDP relaxations.
It is worth mentioning that since the purpose of this paper is to find efficient solutions to \eqref{modelsetting}, the assumption of the positive definiteness of the Hessian of the associated Lagrangian (resp., the objective function) at a saddle-point (resp., an optimal solution) may be theoretical rather than practical.
On the other hand, if the weighted sum polynomial $\lambda^Tf$ is strongly convex, then, by Lemma~\ref{lemma1}, the problem \eqref{hybridmodel} has an optimal solution (in fact, it is unique).
In addition, since the Hessian $\nabla^2(\lambda^Tf)$ of the weighted sum polynomial $\lambda^Tf$ is positive definite on $\R^n,$ we see that Assumption~\ref{assumption2} (ii) holds.
For simplicity, we illustrate our results by an example which satisfies all of the assumptions described above (see, Example~\ref{example1}).}
\end{remark}
\begin{algorithm}[ht]
\begin{algorithmic}
\caption{Finding Efficient Solutions to \eqref{modelsetting}}\label{algorithm1}
\STATE \textbf {Input} :  Fix $\lambda\in{\rm int\,}\R^p_+.$
\STATE \textbf {Step 0}. Set $k=k_0.$
\STATE \textbf {Step 1}. Pick $z\in\R^n$ arbitrarily.
\STATE \textbf {Step 2}. If $K_z = \emptyset,$ then return to Step 1; otherwise$,$ go to Step 3.
\STATE \textbf {Step 3}. Solve  $({\rm Q}^k_z)$ (or $({\rm P}^k_z)$) and obtain its optimal solution $\bar y.$
\STATE \textbf {Step 4}. If the flat truncation condition \eqref{flat truncation} is satisfied, go to Step 5; otherwise, set $k=k+1$ and go back to Step 3.
\STATE \textbf {Step 5}. Extract a unique optimal solution $\bar x$ to problem~\eqref{hybridmodel} from $\bar y.$
\STATE \textbf {Output} : Efficient solution $\bar x$ (by Proposition~\ref{proposition1}).
\end{algorithmic}
\end{algorithm}

We close the section by designing the following example, which illustrates the work of finding efficient solutions to \eqref{modelsetting} with convex polynomial data via Algorithms~\ref{algorithm1}.
\begin{example}\label{example1}{\rm
Consider the following $2$-dimensional multi-objective convex polynomial optimization problem:
\begin{eqnarray*}
({\rm MP})_1
& \min\limits_{(x_1,x_2)\in\R^2} &\big(f_1(x_1,x_2),f_2(x_1,x_2),f_3(x_1,x_2)\big) \\
& {\rm s.t. } & g_1(x_1,x_2)\leq0,\\
&& g_2(x_1,x_2)\leq0,
\end{eqnarray*}
where $f_1(x_1,x_2)=(x_1-3)^2+(x_2-2)^2,$ $f_2(x_1,x_2)=x_1+x_2,$ $f_3(x_1,x_2)=x_1+2x_2,$ $g_1(x_1,x_2)=-x_1,$ and $g_2(x_1,x_2)=-x_2.$
Let $K_1=\{(x_1,x_2)\in\R^2 : -x_1\leq0, \ -x_2\leq0\}=\R^2_+$ be the feasible set of $({\rm MP})_1.$

It is worth noting that the best known set of efficient solutions to $({\rm MP})_1$ is as follows:
\begin{eqnarray*}
\bigg\{(x_1,x_2)\in\R^2&:&\textrm{either} \ 0\leq x_1\leq 0, \ x_2=0 \ \textrm{or} \ x=t_1\left(\begin{array}{c}
  1 \\
  0 \\
\end{array}\right)
+t_2\left(\begin{array}{c}
  2 \\
  0 \\
\end{array}\right)
+t_3\left(\begin{array}{c}
  3 \\
  2 \\
\end{array}\right), \\
&&  \ t_1+t_2+t_3=1, \ t_i\geq0, \ i=1,2,3\bigg\}
\end{eqnarray*}
(see, e.g., \cite[Example~2 in Chapter~6]{Chankong1983}).

Now, consider the following (scalar) polynomial optimization problem with $\lambda=(\lambda_1,\lambda_2,\lambda_3):=(1,1,1),$
\begin{eqnarray*}
({\rm P}_z)_1
& \min\limits_{(x_1,x_2)\in\R^2} &(x_1-3)^2+(x_2-2)^2+2x_1+3x_2 \\
& {\rm s.t. } & -x_1\leq 0, \ -x_2\leq 0, \\
&&(x_1-3)^2+(x_2-2)^2\leq (z_1-3)^2+(z_2-2)^2,\\
&&x_1+x_2\leq z_1+z_2,\\
&&x_1+2x_2\leq z_1+2z_2.
\end{eqnarray*}
Let $(z_1,z_2)\in K_1$ be any given.
Then we see that the Slater-type condition for $({\rm P}_z)_1$ holds except $(z_1,z_2)=(3,2).$
On the other hand, if we choose $(z_1,z_2)=(3,2),$ then the feasible set of $({\rm P}_z)_1$ is $\{(3,2)\}.$
So, in this case, we have $(3,2)$ is an optimal solution to $({\rm P}_z)_1,$ and so is an efficient solution to $({\rm MP})_1.$
Moreover, a simple computation yields that the Hessian $\sum_{j=1}^3\nabla^2f_j$ is positive definite on $\R^2,$
and hence, for all $z:=(z_1,z_2)\in K_1\backslash\{(3,2)\},$ all of the assumptions of Theorem~\ref{theorem8} are satisfied.

On the other hand, for $k\geq1,$ the hierarchy semidefine programming problem, related with $({\rm P}_{z})_1,$ reads as follows
\begin{eqnarray*}
({\rm P}^k_z)_1 \quad
&\inf\limits_{y\in\R^{s(2k)}} & \sum_{\alpha\in\N_{2k}^2}\sum_{j=1}^3(f_j)_\alpha y_\alpha\\
& {\rm s.t. } & M_k(y)\succeq0,\\
&&\sum_{\alpha\in\N^2_{2k}}(g_i)_\alpha y_\alpha\leq0, \  i=1,2,\\
& & \sum_{\alpha\in\N^2_{2k}}(f_j)_\alpha y_\alpha\leq f_j(z), \ j=1,2,3,\\
& & M_{k-1}\big((-\lambda^Tf_z)y\big)\succeq0.
\end{eqnarray*}

Now, let us pick $z=(1,1)\in K_1.$
Then we consider the problem $({\rm P}^k_z)_1$ with $k=1$
\begin{eqnarray*}
({\rm P}^1_{z})_1 \quad
&\inf\limits_{y\in\R^{6}} & 13-4y_{(1,0)}-y_{(0,1)}+y_{(2,0)}+y_{(0,2)}\\
& {\rm s.t. } & M_1(y)=\left(
         \begin{array}{ccc}
           1 & y_{(1,0)} & y_{(0,1)} \\
           y_{(1,0)} & y_{(2,0)} & y_{(1,1)} \\
           y_{(0,1)} & y_{(1,1)} & y_{(0,2)} \\
         \end{array}
       \right)\succeq0,\\
&&-y_{(1,0)}\leq0, \ -y_{(0,1)}\leq0,\\
& & 13-6y_{(1,0)}-4y_{(0,1)}+y_{(2,0)}+y_{(0,2)}\leq 5,\\
&& y_{(1,0)}+y_{(0,1)}\leq2, \ y_{(1,0)}+2y_{(0,1)}\leq3,\\
&& M_{0}((9-\lambda^Tf)y)=-4+4y_{(1,0)}+y_{(0,1)}-y_{(2,0)}-y_{(0,2)}\geq0.
\end{eqnarray*}
Solving $({\rm P}^1_{z})_1$ using GloptiPoly~3 \cite{Henrion2009} yields an optimal value $8.875$ and an optimal solution
\begin{eqnarray*}
\bar y=(1,1.7500,0.2500,3.0625, 0.4375,0.0625).
\end{eqnarray*}
Then, we easily check that ${\rm rank\,}M_1(\bar y)=1={\rm rank\,}M_{0}(\bar y),$ and so, $\bar x=(\bar y_\alpha)_{|\alpha|=1}=(1.75,0.25)$ is an optimal solution to $({\rm P}_{z})_1.$
It follows from Proposition~\ref{proposition1} that $\bar x=(1.75,0.25)$ is an efficient solution to $({\rm MP})_1.$
\begin{figure}[t]
  \vspace{-4mm}
\includegraphics[width=1.0\textwidth]{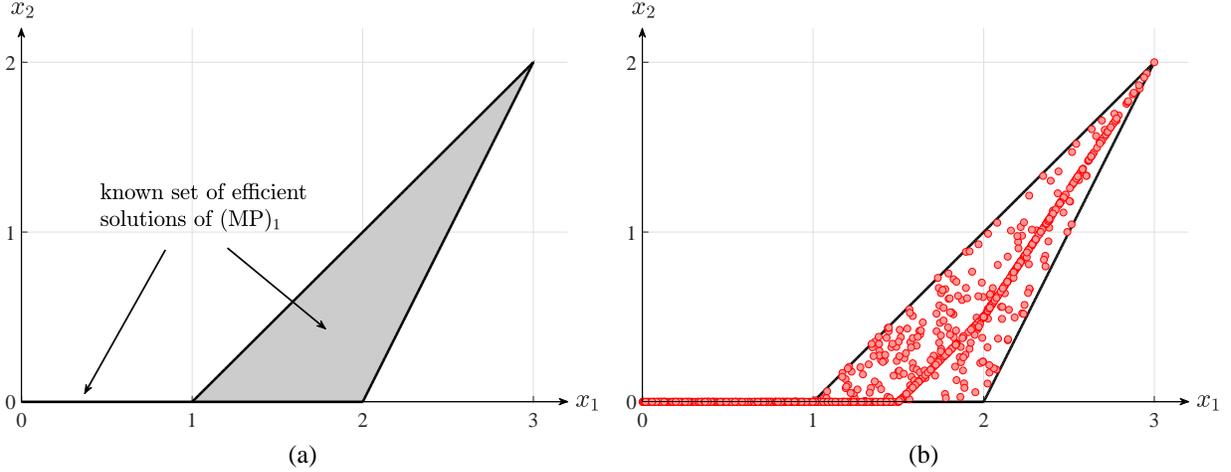}
  \vspace{-12mm}
  \caption{{\bf Efficient solution set comparison.} (a) Efficient solution set to $({\rm MP})_1$ (see, e.g., \cite[Example~2 in Chapter~6]{Chankong1983}). (b) The obtained efficient solutions to $({\rm MP})_1$ are depicted in red.}\label{fig1}
\end{figure}

In order to find more efficient solutions to $({\rm MP})_1,$ we need to parametrically move $z\in K.$
So, we give $1000$ points $z\in K$ arbitrarily to find more efficient solutions and Figure~\ref{fig1} shows that, for the given $1000$ points $z,$ all of obtained efficient solutions to $({\rm MP})_1$ belongs to the known set of efficient solutions to $({\rm MP})_1.$
}\end{example}

\begin{remark}
{\rm Note that from a theoretical point of view, if we move $z$ all around $K$ and {\em fixed} any chosen $\lambda$ from ${\rm int\,}\R^p_+$, then we can get all efficient solutions to problem~\eqref{modelsetting}; see Giannessi et al. \cite[page 191]{Giannessi2000}.
Indeed, in the case of our paper, this conclusion could be guaranteed if we combine Proposition~\ref{proposition1} and the next Proposition~\ref{forallsol.}.}
\end{remark}

\begin{proposition}\label{forallsol.}
  Let $\lambda \in {\rm int\,}\R^p_+$ be fixed.
  Then the following statements are equivalent$:$
  \begin{itemize}
    \item[{\rm (i)}] $\bar x\in K$ is an efficient solution to problem~\eqref{modelsetting}$;$
    \item[{\rm (ii)}] the system {\rm (}in the unknown $x${\rm ):}
  \begin{equation}\label{systemnosolution2}
\left\{
                \begin{array}{lll}
                  \sum_{j=1}^p\lambda_jf_j(x) < \sum_{j=1}^p\lambda_jf_j(\bar x), \\
                  f_j(x) \leq f_j(\bar x),\ \textrm{\rm for all} \ j = 1, \ldots, p, \\
                  g_i(x) \leq 0,\ i = 1, \ldots, m,
                \end{array}
              \right.
\end{equation}
has no solution.
  \end{itemize}
Furthermore$,$ the impossibility of \eqref{systemnosolution2} is a necessary and sufficient condition for $\bar x$ to be an optimal solution to problem~\eqref{hybridmodel}.
\end{proposition}

\section{Conclusions and further remarks} \label{sect:6}

In this paper, we mainly investigated the issue that how to find efficient solutions of a multi-objective optimization problem with {\it convex polynomial data} by using the well-known hybrid method and the SDP relaxation approach.
To this end, an existence result for efficient solutions to \eqref{modelsetting} under some mild assumption was firstly established.
Then, two kinds of representations of non-negativity of convex polynomials over convex semi-algebraic sets were formulated, and two kinds of Lasserre-type hierarchies of SDP relaxations for problem \eqref{hybridmodel} with their finite convergence results were also discussed.
Finally, we showed a scheme on finding efficient solutions to the problem \eqref{modelsetting} by solving hierarchies of semidefinite programming relaxations and checking a flat truncation condition.

It will be very interesting if we consider a class of multi-objective optimization problems with (not necessary convex) polynomials, and study how to find/approximate its (weakly) efficient solutions.
Note that the existence of efficient solutions for (unconstrained) multi-objective polynomial optimization problems has been recently investigated by Kim et al. \cite{Kim2019}.

\subsection*{Acknowledgments}
The authors would like to express their sincere thanks to Prof. Tien-Son Pham of University of Dalat for his valuable suggestions and warm helps.
They also would like to appreciate the anonymous referees for their very helpful and valuable suggestions and comments for the paper.


\begin{thebibliography}{10}

\bibitem{Ahmadi2012}
A.~A. Ahmadi and P.~A. Parrilo.
\newblock A convex polynomial that is not {SOS}-convex.
\newblock {\em Mathematical Programming}, 135(1-2):380--429, 2012.

\bibitem{Ahmadi2013}
A.~A. Ahmadi and P.~A. Parrilo.
\newblock A complete characterization of the gap between convexity and
  {SOS}-convexity.
\newblock {\em SIAM Journal on Optimization}, 23(2):811--833, 2013.

\bibitem{Belousov2002}
E.~G. Belousov and D.~Klatte.
\newblock A {Frank--Wolfe} type theorem for convex polynomial programs.
\newblock {\em Computational Optimization and Applications}, 22(1):37--48,
  2002.

\bibitem{Blanco2014}
V.~Blanco, J.~Puerto, and S.~E. H.~B. Ali.
\newblock A semidefinite programming approach for solving multiobjective linear
  programming.
\newblock {\em Journal of Global Optimization}, 58(3):465--480, 2014.

\bibitem{Chankong1983}
V.~Chankong and Y.~Y. Haimes.
\newblock {\em Multiobjective Decision Making: Theory and Methodology}.
\newblock Amsterdam, North-Holland, 1983.

\bibitem{Ehrgott2005}
M.~Ehrgott.
\newblock {\em Multicriteria Optimization (2nd ed.)}.
\newblock Springer, Berlin, 2005.

\bibitem{Fernando2006}
J.~F. Fernando and J.~M. Gamboa.
\newblock Polynomial and regular images of $\mathbb{R}^n$.
\newblock {\em Israel Journal of Mathematics}, 153:61--92, 2006.

\bibitem{Fernando2014}
J.~F. Fernando and C.~Ueno.
\newblock On complements of convex polyhedra as polynomial and regular images
  of $\mathbb{R}^n$.
\newblock {\em International Mathematics Research Notices},
  2014(18):5084--5123, 2014.

\bibitem{Giannessi2000}
F.~Giannessi, G.~Mastroeni, and L.~Pellegrini.
\newblock On the theory of vector optimization and variational inequalities.
  {Image} space analysis and separation.
\newblock In F.~Giannessi, editor, {\em Vector Variational Inequalities and
  Vector Equilibria: Mathematical Theories. Nonconvex Optimization and Its
  Applications}, volume~38, pages 141--215. Springer, Boston, MA, 2000.

\bibitem{Gorissen2012}
B.~L. Gorissen and D.~d. Hertog.
\newblock Approximating the {P}areto set of multiobjective linear programs via
  robust optimization.
\newblock {\em Operations Research Letters}, 40(5):319--324, 2012.

\bibitem{Ha2017}
H.-V. H\`a and T.~S. Ph\d{a}m.
\newblock {\em Genericity in Polynomial Optimization}.
\newblock World Scientific Publishing, 2017.

\bibitem{Helton2010}
J.~W. Helton and J.~W. Nie.
\newblock Semidefinite representation of convex sets.
\newblock {\em Mathematical Programming}, 122(1):21--64, 2010.

\bibitem{Henrion2005}
D.~Henrion and J.~B. Lasserre.
\newblock Detecting global optimality and extracting solutions in gloptipoly.
\newblock In D.~Henrion and A.~Garulli, editors, {\em Positive Polynomials in
  Control. Lecture Notes in Control and Information Science}, volume 312, pages
  293--310. Springer, Berlin, Heidelberg, 2005.

\bibitem{Henrion2009}
D.~Henrion, J.~B. Lasserre, and J.~Loefberg.
\newblock Gloptipoly 3: moments, optimization and semidefinite programming.
\newblock {\em Optimization Methods and Software}, 24(4-5):761--779, 2009.

\bibitem{Jeyakumar2014}
V.~Jeyakumar, T.~S. Ph\d{a}m, and G.~Li.
\newblock Convergence of the {Lasserre} hierarchy of {SDP} relaxations for
  convex polynomial programs without compactness.
\newblock {\em Operations Research Letters}, 42(1):34--40, 2014.

\bibitem{Jiao2019}
L.~G. Jiao and J.~H. Lee.
\newblock Finding efficient solutions in robust multiple objective optimization
  with {SOS}-convex polynomial data.
\newblock {\em Annals of Operations Research}, 2019.
\newblock \url{https://doi.org/10.1007/s10479-019-03216-z}.

\bibitem{Jiao2019b}
L.~G. Jiao, J.~H. Lee, Y.~Ogata, and T.~Tanaka.
\newblock Multi-objective optimization problems with {SOS}-convex polynomials
  over an {LMI} constraint.
\newblock {\em Taiwanese Journal of Mathematics}, 24(4):1021--1043, 2020.

\bibitem{Jiao2020}
L.~G. Jiao, J.~H. Lee, and Y.~Y. Zhou.
\newblock A hybrid approach for finding efficient solutions in vector
  optimization with {SOS}-convex polynomials.
\newblock {\em Operations Research Letters}, 48(2):188--194, 2020.

\bibitem{Kim2018}
D.~S. Kim, B.~S. Mordukhovich, T.~S. Ph\d{a}m, and N.~V. Tuyen.
\newblock Existence of efficient and properly efficient solutions to problems
  of constrained vector optimization.
\newblock {\em Mathematical Programming}, 2020.
\newblock \url{https://doi.org/10.1007/s10107-020-01532-y}.

\bibitem{Kim2019}
D.~S. Kim, T.~S. Ph\d{a}m, and N.~V. Tuyen.
\newblock On the existence of {P}areto solutions for polynomial vector
  optimization problems.
\newblock {\em Mathematical Programming}, 177(1-2):321--341, 2019.

\bibitem{Lasserre2001}
J.~B. Lasserre.
\newblock Global optimization with polynomials and the problem of moments.
\newblock {\em SIAM Journal on Optimization}, 11(3):796--817, 2001.

\bibitem{Lasserre2009}
J.~B. Lasserre.
\newblock Convexity in semialgebraic geometry and polynomial optimization.
\newblock {\em SIAM Journal on Optimization}, 19(4):1995--2014, 2009.

\bibitem{Lasserre2010}
J.~B. Lasserre.
\newblock {\em Moments, Positive Polynomials and their Applications}.
\newblock Imperial College Press, London, 2010.

\bibitem{Lasserre2015}
J.~B. Lasserre.
\newblock {\em An Introduction to Polynomial and Semi-Algebraic Optimization}.
\newblock Cambridge University Press, 2015.

\bibitem{Lee2018}
J.~H. Lee and L.~G. Jiao.
\newblock Solving fractional multicriteria optimization problems with sum of
  squares convex polynomial data.
\newblock {\em Journal of Optimization Theory and Applications},
  176(2):428--455, 2018.

\bibitem{Lee2019}
J.~H. Lee and L.~G. Jiao.
\newblock Finding efficient solutions for multicriteria optimization problems
  with {SOS}-convex polynomials.
\newblock {\em Taiwanese Journal of Mathematics}, 23(6):1535--1550, 2019.

\bibitem{Luc2016}
D.~T. Luc.
\newblock {\em Multiobjective Linear Programming: An Introduction}.
\newblock Springer International Publishing, Switzerland, 2016.

\bibitem{Magron2014}
V.~Magron, D.~Henrion, and J.~B. Lasserre.
\newblock Approximating {P}areto curves using semidefinite relaxations.
\newblock {\em Operations Research Letters}, 42(6-7):432--437, 2014.

\bibitem{Nie2013}
J.~Nie.
\newblock Certifying convergence of {Lasserre's} hierarchy via flat truncation.
\newblock {\em Mathematical Programming}, 142(1-2):485--510, 2013.

\bibitem{Nie2009}
J.~Nie and K.~Ranestad.
\newblock Algebraic degree of polynomial optimization.
\newblock {\em SIAM Journal on Optimization}, 20(1):485--502, 2009.

\bibitem{Pham2019}
T.~S. Ph\d{a}m.
\newblock Optimality conditions for minimizers at infinity in polynomial
  programming.
\newblock {\em Mathematics of Operations Research}, 44(4):1381--1395, 2019.

\bibitem{Ploskas2017}
N.~Ploskas and N.~Samaras.
\newblock {\em Linear Programming using MatLab}.
\newblock Springer International Publishing AG, Switzerland, 2017.

\bibitem{Putinar1993}
M.~Putinar.
\newblock Positive polynomials on compact semi-algebraic sets.
\newblock {\em Indiana University Mathematics Journal}, 42(3):969--984, 1993.

\bibitem{Putinar1999}
M.~Putinar and F.-H. Vasilescu.
\newblock Positive polynomials on semi-algebraic sets.
\newblock {\em Comptes Rendus de l'Acad\'emie des Sciences - Series I -
  Mathematics}, 328(7):585--589, 1999.

\bibitem{Scheiderer2003}
C.~Scheiderer.
\newblock Sums of squares on real algebraic curves.
\newblock {\em Mathematische Zeitschrift}, 245(4):725--760, 2003.

\bibitem{Scheiderer2005}
C.~Scheiderer.
\newblock Distinguished representations of non-negative polynomials.
\newblock {\em Journal of Algebra}, 289(2):558--573, 2005.

\bibitem{Schmudgen1991}
K.~Schm\"{u}dgen.
\newblock The {$K$-moment} problem for compact semi-algebraic sets.
\newblock {\em Mathematische Annalen}, 289(2):203--206, 1991.

\bibitem{Schweighofer2005}
M.~Schweighofer.
\newblock Optimization of polynomials on compact semialgebraic sets.
\newblock {\em SIAM Journal on Optimization}, 15(3):805--825, 2005.

\bibitem{Ueno2008}
C.~Ueno.
\newblock A note on boundaries of open polynomial images of $\mathbb{R}^2$.
\newblock {\em Revista Matematica Iberoamericana}, 24(3):981--988, 2008.

\end{thebibliography}
\end{document}